\documentclass[11pt]{amsart}

\usepackage{amssymb,amscd,amsthm,amsxtra}
\usepackage{latexsym}

\newtheorem{thm}{Theorem}[section]

\newtheorem{lem}[thm]{Lemma}
\newtheorem{prop}[thm]{Proposition}
\theoremstyle{definition}
\newcommand{\comment}[1]{}

\newtheorem{defn}[thm]{Definition}
\theoremstyle{remark}
\newtheorem{rem}[thm]{Remark}
\numberwithin{equation}{section}

\newcommand{\erf}{{\mbox{erf}}}

\newcommand{\ir}{I\times\R}

\newcommand{\zit}{Z_I(t)}

\newcommand{\itxi}{\int_{t_0}^t\int_{\sum_{i=1}^{2k+2}\xi_i=0}}
\newcommand{\symb}
{1-\frac{m(\xi_2+\xi_3+\cdots+\xi_{2k+2})}{m(\xi_2)m(\xi_3)\cdots m(\xi_{2k+2})}}

\newcommand{\hs}{H_x^s}
\newcommand{\dhs}{\dot H_x^s}
\newcommand{\ho}{\dot H_x^1}
\newcommand{\ulam}{u^{\lambda}}
\newcommand{\lxt}{L_x^2}

\newcommand{\R}{\mathbb R}
\newcommand{\C}{\mathbb C}

\begin{document}

% \title[Correlation Estimates and Applications.]
% {\bf Commutator vector operators, Correlation Estimates and applications to Nonlinear Schr\"odinger equations.}
\title[Correlation Estimates and Applications.]{Tensor products and Correlation Estimates with applications to Nonlinear Schr\"odinger equations.}
\author{J. Colliander}
\thanks{J.C. was supported in part by NSERC grant RGP250233-07.}
\address{Department of Mathematics, University of Toronto, Toronto, ON, Canada M5S 2E4}
\email{\tt colliand@math.toronto.edu}

\author{M. Grillakis}
\address{Department of Mathematics, University of Maryland,
College Park, MD, 20742} \email{\tt mng@math.umd.edu}

\author{N. Tzirakis}
\address{Department of Mathematics, University of Illinois at Urbana-Champaign, Urbana, IL, 61801}
\email{tzirakis@math.uiuc.edu}
\date{November 28, 2007}

\subjclass{}

\keywords{}
\begin{abstract}
We prove new interaction Morawetz type (correlation) estimates in one and two dimensions. In dimension two 
the estimate corresponds to the nonlinear diagonal analogue of Bourgain's bilinear refinement of Strichartz. For the 2d case we 
provide a proof in two different ways.
 First, we follow the original approach of Lin and Strauss but applied
 to tensor products of solutions. We then demonstrate the proof  
using commutator vector operators acting on the conservation laws of the equation. This method can be generalized to obtain
 correlation estimates in all dimensions. In one dimension we use the Gauss-Weierstrass summability method acting on the conservation laws.
 We then apply the 2d estimate to 
nonlinear Schr\"odinger equations and derive a direct proof of
Nakanishi's $H^1$ scattering result  
for every $L^{2}$-supercritical nonlinearity. We also prove scattering
below the energy space for a certain class of $L^{2}$-supercritical  
equations.   
\end{abstract}
\maketitle

\section{Introduction}
In this paper we obtain new{\footnote{The same estimates have been
    independently and simultaneously (see \cite{col} and \cite{pv1})
obtained by F. Planchon and L. Vega \cite{pv} with different proofs.}} a priori estimates for solutions of the
nonlinear Schr\"odinger equation in one and two dimension. We also provide
 a systematic way to obtain the known interaction a priori estimates for dimensions higher than three. 
These estimates are monotonicity formulae that take advantage of the conservation of the momentum of the equation. 
Due to the pioneering work \cite{cm}, estimates of this type are
referred to as {\it{Morawetz estimates}}  in the literature. We then
apply these estimates to study the global behavior
 of solutions to the nonlinear Schr\"odinger equation. 
To be more precise we want to study
 the global-in-time behavior of solutions to the following initial value problem
\begin{equation}\label{nls}
\left\{
\begin{matrix}
iu_{t}+ \Delta u -|u|^{p-1}u=0, & x \in {\mathbb R^n}, & t\in {\mathbb R},\\
u(x,0)=u_{0}(x)\in H^{s}({\mathbb R^n}).
\end{matrix}
\right.
\end{equation}
with $p>1.$ Here we investigate the $L^2$-supercritical equation in
two dimensions under the natural scaling of the equation,  
and thus we restrict $p$ to $p>3$. Scaling refers to the fact that if
$u(x,t)$ is a solution to \eqref{nls} then  
$$u^{\lambda}(x,t)=\lambda^{-\frac{2}{p-1}}u(\frac{x}{\lambda},\frac{t}{\lambda^{2}})$$ is also a solution. The problem is then called
 $H^{s}$-critical if the scaling leaves the homogeneous $\dot{H}^{s}$
 norm invariant.  
This happens exactly when $s=\frac{n}{2}-\frac{2}{p-1}$. We denote the
critical index by $s_{c}$ and thus
 \begin{equation}\label{index}
s_c=\frac{n}{2}-\frac{2}{p-1}.
\end{equation}
\\
\\
The problem of the existence of local-in-time solutions for \eqref{nls} 
is well studied by many authors and a summary of the results can been found in \cite{jb2}, \cite{tc}, 
and \cite{tt}. Thus depending on the strength of the nonlinearity and
the dimension, the local solutions are well understood. In this paper
we will consider problems that are locally well-posed and refer the
reader to \cite{tc},  
and \cite{tt} for the proofs. 
The local well-posedness definition that we use here reads as
follows: for any choice of initial data $u_0 \in H^s$, there exists
a positive time $T = T(\|u_0\|_{H^{s}})$ depending only on the
norm of the initial data, such that a solution to the initial
value problem exists on the time interval $[0,T]$, it is unique in
a certain Banach space of functions $X\subset C([0,T],H^s_{x})$,
and the solution map from $H^s_{x}$ to $C([0,T],H^s_{x})$ depends
continuously on the initial data on the time interval $[0,T]$. If
the time $T$ can be proved to be arbitrarily large, we say that the Cauchy problem is globally well-posed. 
To extend a local solution to a global one, we need some a priori information about the
 norms of the solution. This usually comes from conservation laws. For  example solutions of 
equation \eqref{nls} 
satisfy mass conservation
\\
\begin{equation}
\|u(t)\|_{L^{2}}=\|u_{0}\|_{L^{2}}\label{mass}
\end{equation}
\\
and smooth solutions also satisfy energy conservation
\\
\begin{equation}\label{energy}
E(u)(t)=\frac{1}{2}\int |\nabla u(t)|^{2}dx+\frac{1}{p+1}\int |u(t)|^{p+1}dx=E(u_{0}). 
\end{equation}
\\
These two conservation laws identify $H^{1}$ and $L^{2}$ as  important spaces concerning the initial value problem \eqref{nls}. We can use them
 to extend the local solutions for all times. For example based on 
energy conservation we immediately get that for initial data $u(t_{0})=u_{0} \in H^{1}$ we have that $\|u(t)\|_{H^{1}} \leq C(u_{0},t_{0})$ for all times.
In general assume that we have an a priori estimate of the form
\\
$$\|u(t)\|_{H^{s}} \leq C(u_{0},t_{0}).$$
\\
In order to use this information to iterate the local solutions, 
the time of local resolution $T$, has to be estimated from below in terms of the norms of the initial data in
 $H^s$, $T \geq M(\|u_{0}\|_{H^s})$, for some strictly positive and non increasing function $M$. This is not the case for the $L^2$ norm of the 
$L^2-$critical problem
which corresponds to the case of $p=1+\frac{4}{n}$, since the local time depends not only on the norm of the initial data but also on the profile.
 On the other hand since the equation \eqref{nls} is energy subcritical in dimensions one and two for any $p$ we have that $T \geq M(\|u_{0}\|_{H^1})$. 
Thus one can iterate the local resolution
 and solve the Cauchy problem at time $t_{k-1}$ ($1\leq k < \infty$) with initial data $u(t_{k-1})$ up to time $t_{k}=t_{k-1}+T_{k}$ with local time
 $T_{k} \geq M(\|u(t_{k-1})\|_{H^1})$. Now if the series $\sum T_{k}$ converges, then on one hand $T_{k}$ tends to zero, but on the other hand 
$T_{k} \geq M(C(u_{0},t_{0},I))$ where $I=[t_{0},t_{0}+\sum T_{k}]$ which is a contradiction. 
Thus the series $\sum T_{k}$ diverges and $u$ can be continued for all times in $H^1$. 
\\
\\
In the situation that the Cauchy problem is globally well-posed, we can address the question of describing and classifying the asymptotic behavior in time
 for global solutions. A possible method to attack the question is to compare the given dynamics with suitably chosen simpler asymptotic dynamics.
 The method applies to a wide variety of dynamical systems and in
 particular to some systems defined by nonlinear PDE, and give rise to
 the scattering theory. 
 For the semilinear problem \eqref{nls}, the first obvious candidate is the free dynamics generated by the group $S(t)=e^{it\Delta}$. The comparison
 between the two dynamics gives rise to the following two questions.
\\
\\
(1) Let $v_{+}(t)=S(t)u_{+}$ be the solution of the free equation. Does there exist a solution $u$ of equation \eqref{nls} which behaves asymptotically
 as $v_{+}$ as $t \rightarrow \infty$, typically in the sense that for a Banach space $X$
\begin{equation}\label{scat}
\|u(t)-v_{+}\|_{X} \rightarrow 0, \ \ \mbox{when $t \rightarrow \infty$.}
\end{equation}   
If this is true then one can define the map $\Omega_{+}: u_{+} \rightarrow u(0)$. The map is called the wave operator and
 the problem of existence of $u$ for given $u_{+}$ is referred to as the problem of the {\it{existence of the wave operator}}. The analogous problem
 arises as $t \rightarrow -\infty$.
\\
\\
(2) Conversely, given a solution $u$ of \eqref{nls}, does there exist an asymptotic state $u_{+}$ such that $v_{+}(t)=S(t)u_{+}$ behaves 
asymptotically as $u(t)$, typically in the sense of \eqref{scat}. If that is the case for any $u$ with initial data in $X$ for some $u_{+} \in X$,
 one says that {\it{asymptotic completeness}} holds in $X$.

Asymptotic completeness is a much harder problem than the existence of the wave operators except in the case of small data theory which follows
 pretty much from the iteration method proof of the local well-posedness. Asymptotic completeness requires a repulsive nonlinearity and
 usually proceeds through the derivation of a priori estimates for general solutions. As we have already mentioned, these estimates take advantage of the momentum conservation law
\begin{equation}
\vec{p}(t)=\Im \int_{\Bbb R^n}\bar{u}\nabla u dx.
\end{equation}
\\
We can establish for example the generalized virial inequality{\footnote{In fact, one can write an identity.}}
 \cite{ls},
\begin{equation}\label{linstr}
\\
\int_{0}^{T}\int_{\Bbb R^n}(-\Delta \Delta a(x))|u(x,t)|^2dxdt+\frac{2(p-1)}{p+1}\int_{0}^{T}\int_{\Bbb R^n}2 \Delta a {|u(x,t)|^{p+1}} dxdt \lesssim
 \sup_{[0,T]}|M_{a}(t)|
\end{equation}
\\
where $a(x)$ is a convex function, $u$ is a solution to \eqref{nls} and $M_{a}(t)$ is the Morawetz action defined by
\\
\begin{equation}\label{mora}
M_{a}(t)=2\int_{\Bbb R^n}\nabla a \cdot \Im(\bar{u}(x)\nabla u(x))dx
\end{equation}
\\
One can use this identity as a starting point and derive a priori 
interaction Morawetz inequalities. These estimates can be achieved by translating the origin in the integrands of (1.7) to an arbitrary 
point $y$ and then averaging \cite{ckstt4} against the $L^1$ mass
density $|u(y)|^2dy$, or by considering{\footnote{This idea emerged in
    a conversation between Andrew Hassell and Terry Tao.}} the tensor
product of two solutions of  
\eqref{nls} and use the fact that the operation of tensoring the two solutions, results again in a defocusing nonlinearity. Both of these methods
 depends on the fact that for dimension $n \geq 3$ the distribution $-\Delta \Delta |x|$ is positive. The estimate one can obtain for $n \geq 3$
 is 
\\
\begin{equation}\label{interac}
\|D^{-\frac{n-3}{2}}(|u|^2)\|_{L_{t}^2L_{x}^{2}} \lesssim \|u\|_{L_{t}^{\infty}\dot{H}_{x}^{\frac{1}{2}}}\|u\|_{L_{t}^{\infty}L_x^2}
\end{equation}
\\
For $n=3$ this estimate reduces to
\\
\begin{equation}\label{mor3}
\|u\|_{L_t^4L_x^4}^2 \lesssim \|u\|_{L_{t}^{\infty}\dot{H}_{x}^{\frac{1}{2}}}\|u\|_{L_{t}^{\infty}L_x^2}.
\end{equation}
\\
This estimate is historically the first interaction Morawetz estimate
and was obtained in \cite{ckstt4}. For $n \geq 4$ it was derived in 
\cite{vis1}, \cite{mv}. The estimate in three dimensions has important consequences. It can be used
 to prove scattering in the energy space for the 3d problem for any
 $p-1>\frac{4}{3}$. This result was obtained in \cite{gv1}, but the
 estimate \eqref{mor3}
 gives a very short and elegant proof. One can also combine this
 estimate with the ``$I$-method'' to show \cite{ckstt4} 
global well-posedness and scattering  to the 3d cubic nonlinear
Schr\"odinger equation below the energy space.    

For solutions below the energy threshold the first result of global
well-posedness was 
established in \cite{jb1} by decomposing the initial data into low frequencies and high frequencies and 
estimating separately the evolution of low and high frequencies. The
key observation was that the high frequencies behave ``essentially
unitarily''. The
method was applied to the cubic equation in two dimensions and 
established that the solution is globally well-posed with
 initial data in $H^s(\Bbb R^2)$ for any $s>\frac{3}{5}$. Moreover if we denote with $S_{t}$ the nonlinear flow and with 
$S(t)=e^{it\Delta}u_{0}$ the linear group, the high/low frequency method 
shows in addition that $\left( S_t-S(t)\right)u_{0} \in H^{1}(\Bbb R^2)$ 
for all times provided $u_0 \in H^s, s > \frac{3}{5}$. 
% A refinement of the high/low decomposition (without though the additional $H^1$ regularity
%  property for the nonlinear part of the solutions) was invented by J. Colliander, M. Keel, G.
% Staffilani, H. Takaoka and T. Tao.
% This method which in the literature is referred to as the $I$-method was introduced  first by Keel and Tao in the context of the wave equation 
% and then further developed by J. Colliander, M. Keel, G.
% Staffilani, H. Takaoka and T. Tao in a series of papers(see, for
% example \cite{ckstt4} and the references therein).
Inspired by
\cite{jb1}, the $I$-method (see \cite{ckstt4} and references therein) 
is based on the almost conservation
of a certain modified energy functional. The idea is to replace
the conserved quantity $E(u)$ which is no longer available for
$s<1$, with an ``almost conserved'' variant $E(Iu)$ where $I$ is a
smoothing operator of order $1-s$ which behaves like the identity
for low frequencies and like a fractional integral operator for
high frequencies. Thus, the operator $I$ maps $H_{x}^{s}$ to
$H_{x}^{1}$. Notice that $Iu$ is not a solution to \eqref{nls}
and hence we expect an energy increment. This increment is in fact
quantifying $E(Iu)$ as an ``almost conserved'' energy. The key is
to prove that on intervals of fixed length, where local
well-posedness is satisfied, the increment of the modified energy
$E(Iu)$ decays with respect to a large parameter $N$. (For the
precise definition of $I$ and $N$ we refer the reader to Section
$2$.) This requires delicate estimates on the commutator between
$I$ and the nonlinearity.

In addition to the $H^1$ scattering problem, a frequency localized version
 of \eqref{mor3} was a main ingredient in the proof that the $\dot{H}^1$-critical NLS is globally well-posed and scatters
 in 3d, \cite{ckstt5}. Note that if \eqref{interac} were true for $n=2$ we would have
\\
\begin{equation}\label{interac2}
\|D^{\frac{1}{2}}(|u|^2)\|_{L_{t}^2L_{x}^{2}} \lesssim \|u\|_{L_{t}^{\infty}\dot{H}_{x}^{\frac{1}{2}}}\|u\|_{L_{t}^{\infty}L_x^2}.
\end{equation} 
\\
This estimate can be consider as the diagonal, nonlinear analogue of
the bilinear refinement of 
Strichartz in \cite{jb1}, and would have many interesting
 applications. A weaker local-in-time estimate was recently obtained \cite{fg}:
\\
\begin{equation}\label{Mor1}
\|u\|_{L^4_{T}L^4_x}^{2} \lesssim
T^{\frac{1}{4}}\|u_0\|_{L^2_x}\|u\|_{L^\infty_T\dot{H}^{\frac{1}{2}}}.
\end{equation} 
\\
This estimate is very useful since the $L_t^{4}L_x^4$ norm is a Strichartz norm and can help one to get a global
 solution assuming the control on the local norms. 
Note the restriction that $u$ has to be at least as regular as an $H^{1/2}$ solution. 
This estimate was recently improved \cite{cgt} to
\\
\begin{equation}\label{newMor1}
\|u\|_{L^4_{T}L^4_x}^{2} \lesssim
T^{\frac{1}{6}}\|u_0\|_{L^2_x}^{\frac{4}{3}}\|u\|_{L^\infty_T\dot{H}^{\frac{1}{2}}}^{\frac{2}{3}}.
\end{equation}
\\
This a priori estimate along with the $I$-method was used 
to establish global well-posedness for the cubic nonlinear Schr\"odinger equation in 2d for any $s>\frac{2}{5}$. Note that these refinements suggest
 the global Strichartz estimate which would immediately imply for $\theta =0$, global well-posedness and scattering for the $L^2$-critical problem
\\
\begin{equation}\label{newMor2}
\|u\|_{L^4_{T}L^4_x}^{2} \lesssim
T^{\frac{\theta}{2}}\|u_0\|_{L^2_x}^{2(1-\theta)}\|u\|_{L^\infty_T\dot{H}^{\frac{1}{2}}}^{2\theta}.
\end{equation}
\\
Unfortunately an argument in \cite{fg} shows that using the above
methods, estimate \eqref{newMor1} is the best possible.

A byproduct of our analysis in \cite{cgt} provides a new estimate in one dimension which reads
\\
\begin{equation}\label{1dmor}
\|u\|_{L^6_{T}L^6_x} \lesssim
T^{\frac{1}{6}}\|u_0\|_{L^2_x}\|u\|_{L^\infty_T\dot{H}^{\frac{1}{2}}}^{\frac{1}{3}}.
\end{equation}
\\
This estimate was used  to prove \cite{dps} global well-posedness for
the 1d $L^2$-critical problem
for any $s>\frac{1}{3}$. Note that for all the above problems the solution is below the $H^{1/2}$ threshold and the a priori estimates are not
 applicable. One has to introduce a smooth cut-off of the initial data and control certain error terms using multilinear harmonic analysis techniques. 
\\
\\
In this paper we prove that \eqref{interac2} is indeed
 true. It is proved by refining the tensor product approach that we
 mentioned above. Using Sobolev embedding, an immediate consequence
 of \eqref{interac2} is the following
\\
\begin{equation}\label{2dmor}
\|u\|_{L^4_{t}L^8_x}^{2} \lesssim
\|u_0\|_{L^2_x}\|u\|_{L^\infty_T\dot{H}^{\frac{1}{2}}}.
\end{equation}
\\
One can use this estimate to obtain a simplified proof of the $H^{1}$ scattering
result in \cite{kn}, in two dimensions for any $p>3$ which avoids the
induction on energy argument and produces a better bound on the spacetime size of the solution.
 For completeness
 we present the proof in Section 4.

We now state the main Theorems of this paper. The estimates contained in Theorems 1.1 and 1.2
below were simultaneously and independently obtained (\cite{pv1}, \cite{pv}) by Planchon and Vega.
\begin{thm}[Correlation estimate in two dimensions]
\label{thm1}
 Let $u$ be an $H^{\frac{1}{2}}$ solution to \eqref{nls} on the spacetime slab $I\times {\Bbb R}^2$. Then
\begin{equation}\label{22d}
\|D^{\frac{1}{2}}(|u|^2)\|_{L_{t}^2L_{x}^{2}} \lesssim \|u\|_{L_{t}^{\infty}\dot{H}_{x}^{\frac{1}{2}}}\|u\|_{L_{t}^{\infty}L_x^2}
\end{equation}
\end{thm}
\begin{thm}[Correlation estimates in one dimension]
\label{thm2}
 Let $u$ be an $H^{1}$ solution to \eqref{nls} on the spacetime slab $I\times \Bbb R$. Then
\begin{equation}\label{11d}
\|\partial_x(|u|^2)\|_{L_{t}^2L_{x}^{2}} \lesssim \|u\|_{L_{t}^{\infty}\dot{H}_{x}^{1}}^{\frac{1}{2}}\|u\|_{L_{t}^{\infty}L_x^2}^{\frac{3}{2}}
\end{equation}
and
\begin{equation}\label{1db}
\|u\|_{L_t^{p+3}L_x^{p+3}}^{p+3}\lesssim \|u\|_{L_t^{\infty}L_x^2}^3\|u\|_{L_t^{\infty}\dot{H}_x^{1}}
\end{equation}
\end{thm}
\begin{thm}[Asymptotic completeness in $H^{1}(\Bbb R^2)$]
\label{thm3}
 Let $u_0\in H^1(\Bbb R^2)$.  Then, there exists a unique global solution $u$ to the initial value problem
\begin{equation}\label{nls1}
\begin{cases}
i u_t +\Delta u = |u|^{p-1}u, \quad p>1,\\
u(0,x) = u_0(x).
\end{cases}
\end{equation}
Moreover, if $p>3$ there exist $u_\pm\in H^1(\Bbb R^2)$ such that
$$
\|u(t)-e^{it\Delta}u_\pm\|_{H^1(\Bbb R^2)}\to 0 \quad \text{as } t\to \pm \infty.
$$
\end{thm}
\begin{thm}[Asymptotic completeness below $H^{1}(\Bbb R^2)$]
\label{thm4}
 Let $u_0\in H^s(\Bbb R^2)$.  Then, 
for each positive integer $k \geq 2$, there exists a regularity threshold $s_k=1-\frac{1}{4k-3}$ such that the following initial value problem
\begin{equation}\label{nlsk}
\begin{cases}
i u_t +\Delta u = |u|^{2k}u, \quad k \geq 2,\\
u(0,x) = u_0(x).
\end{cases}
\end{equation}
is globally well-posed and scatters provided $s>s_k$. In particular there exists $u_\pm\in H^s(\Bbb R^2)$ such that
$$
\|u(t)-e^{it\Delta}u_\pm\|_{H^s(\Bbb R^2)}\to 0 \quad \text{as } t\to \pm \infty.
$$
\end{thm}
We note that the estimates \eqref{11d} and \eqref{22d} come from the linear part of the solution and thus are true for any nonlinearity, while
 estimate \eqref{1db} comes form the nonlinear part. Actually the proof of Theorem 1 shows that the following estimate is true for any $n \geq 2$
(with the appropriate interpretations of course when the power of the derivative operator is positive or negative)

$$\|D^{-\frac{n-3}{2}}(|u|^2)\|_{L_{t}^2L_{x}^{2}} \lesssim \|u\|_{L_{t}^{\infty}\dot{H}_{x}^{\frac{1}{2}}}\|u\|_{L_{t}^{\infty}L_x^2}.$$

The basic idea behind these new estimates is to view
 the evolution equations as describing the evolution of a compressible dispersive fluid whose pressure is a function of the density. In this case
 the mass and momentum conservation laws describe the conservation laws of an irrotational compressible and dispersive fluid. There is a difference
 though in one and two dimensions. In two and higher dimensions we use
 commutator vector operators that 
 act on the conservation laws. In dimension one we use the heat
 kernel. 
More precisely, we introduce into the Morawetz action
 the error function 
\\
$$\erf(x)=\int_{-\infty}^{x}e^{-t^2}dt$$ 
\\
scaled by $\epsilon$ whose derivative is the heat kernel in one
dimension. We define 
 the operator that is given as a convolution with the error function and apply it to the conservation laws of the equation. Integration by parts
produces the solution of the one dimensional heat equation. Sending $\epsilon$ to zero we recover the estimates. This way the mass
 density plays the role of the initial data of the linear heat equation and the method is nothing else than the Gauss-Weierstrass summability
 method in classical Fourier Analysis. Again for details the reader can consult Section 4.  
\\

The rest of the paper is organized as follows. In Section 2 we
introduce some notation and state important propositions that we
will use throughout the paper. In Section 3 we present the proofs of the correlation estimates in all dimensions
 and provide a general framework for obtaining similar estimates. In
 Section 4 we prove the $H^{1}$ scattering result for the
 $L^2$-supercritical nonlinear Schr\"odinger in two dimensions
 (Theorem \ref{thm3}).  
Finally in Section 5 we prove global well-posedness and scattering below the energy space of the initial value problem \eqref{nlsk} (Theorem \ref{thm4}.)

\section{Notation}
In this section, we introduce notations and some basic estimates we will invoke throughout this paper. We use 
$A \lesssim B$ to denote an estimate of the form $A\leq CB$ for some constant $C$.
If $A \lesssim B$ and $B \lesssim A$ we say that $A \sim B$. We write $A \ll B$ to denote
an estimate of the form $A \leq cB$ for some small constant $c>0$. In addition $\langle a \rangle:=1+|a|$ and
$a\pm:=a\pm \epsilon$ with $0 < \epsilon \ll 1$.

We use $L_x^r(\Bbb R^n)$ to denote the Banach space of functions $f:\Bbb R^n\to \C$ whose norm
\\
$$\|f\|_r:=\Bigl(\int_{\Bbb R^n} |f(x)|^r dx \Bigr)^{1/r}$$
\\
is finite, with the usual modifications when $r=\infty$.

We use $L_t^qL_x^r$ to denote the spacetime norm
\\
$$\|u\|_{q,r}:=\|u\|_{L_t^qL_x^r(\Bbb R\times\Bbb R^n)}:=\Bigl(\int_{\Bbb R}\Bigl(\int_{\Bbb R^n} |u(t,x)|^r dx \Bigr)^{q/r}dt\Bigr)^{1/q},$$
\\
with the usual modifications when either $q$ or $r$ are infinity, or when the domain $\Bbb R \times \Bbb R^n$ is replaced by some smaller
spacetime region.  When $q=r$ we abbreviate $L_t^qL_x^r$ by $L^q_{t,x}$. We define the Fourier transform of $f(x) \in L_{x}^{1}(\Bbb R^n)$ by
\\
$$\hat{f}(\xi)=\int_{\mathbb R^n} e^{-2\pi i\xi x}f(x)dx.$$
\\
For an appropriate class of functions the following Fourier inversion formula holds:
\\
$$f(x)=\int_{\mathbb R^n} e^{2\pi i\xi x}\hat{f}(\xi)(d\xi).$$
\\
Moreover we know that the following identities are true:\\
\begin{enumerate}
\item $\|f\|_{L^{2}}=\|\hat{f}\|_{L^{2}}$, (Plancherel)\\
\item $\int_{\mathbb R^n} f(x)\bar{g}(x)dx=\int_{\mathbb R^n} \hat{f}(\xi)\bar{\hat{g}}(\xi) (d\xi)$, (Parseval)\\
\item $\widehat{fg}(\xi)=\hat{f}\star \hat{g}(\xi)=\int_{\mathbb R^n} \hat f(\xi-\xi_{1})\hat g(\xi_{1})d\xi_{1}$, (Convolution).\\
\end{enumerate}
We will also make use of the fractional differentiation operators $|\nabla|^s$ defined by
$$
\widehat{|\nabla|^sf}(\xi) := |\xi|^s \hat f (\xi).
$$
These define the homogeneous Sobolev norms
$$
\|f\|_{\dot H^{s}_x} := \| |\nabla|^s f \|_{L^2_x}
$$
and more general Sobolev norms
$$
\|f\|_{H_x^{s,p}}:=\|\langle\nabla\rangle^s f\|_p,
$$
where, $\langle\nabla\rangle=(1+|\nabla|^2)^{\frac 12}$.
\\
\\
Let $e^{it\Delta}$ be the free Schr\"odinger propagator.  In physical space this is given by the formula
\\
$$e^{it\Delta}f(x) = \frac{1}{(4 \pi i t)^{\frac{n}{2}}} \int_{\Bbb R^n} e^{i|x-y|^2/4t} f(y) dy$$
\\
for $t\neq 0$ (using a suitable branch cut to define $(4\pi it)^{\frac{d}{2}}$), while in frequency space one can write this as
\\
\begin{equation}\label{fourier rep}
\widehat{e^{it\Delta}f}(\xi) = e^{-4 \pi^2 i t |\xi|^2}\hat f(\xi).
\end{equation}
\\
In particular, the propagator obeys the \emph{dispersive inequality}
\\
\begin{equation}\label{dispersive ineq}
\|e^{it\Delta}f\|_{L^\infty_x} \lesssim |t|^{-\frac{n}{2}}\|f\|_{L^1_x}
\end{equation}
for all times $t\neq 0$.
\\
\\
We also recall \emph{Duhamel's formula}
\begin{align}\label{duhamel}
u(t) = e^{i(t-t_0)\Delta}u(t_0) - i \int_{t_0}^t e^{i(t-s)\Delta}(iu_t + \Delta u)(s) ds.
\end{align}

\begin{defn}
A pair of exponents $(q,r)$ is called \emph{Schr\"odinger-admissible} if $(q,r,n)\ne (2,\infty,2)$
$$
\frac{2}{q} +\frac{n}{r} = \frac{n}{2}, \quad  2 \leq r \leq \infty.
$$
\end{defn}
For a spacetime slab $I\times\Bbb R^n$, we define the Strichartz norm
\\
$$\|f\|_{S^0(I)}:=\sup_{(q,r)\text{ admissible}}\|f\|_{L_t^qL_x^r(I\times\Bbb R^n)}.$$
Then, we have the following Strichartz estimates (for a proof see \cite{kt} and the references therein):

\begin{lem}\label{linstr}
Let $I$ be a compact time interval, $t_0\in I$, $s\ge 0$, and let $u$ be a solution to the forced Schr\"odinger equation
\begin{equation*}
i u_t + \Delta u =\sum_{i=1}^m F_i
\end{equation*}
for some functions $F_1, \dots, F_m$.  Then,
\begin{equation}
\||\nabla|^s u\|_{ S^0(I)} \lesssim \|u(t_0)\|_{\dot{H}_x^s} + \sum_{i=1}^m \||\nabla| ^s F_i\|_{L_t^{q_i'}L_x^{r_i'}(I\times\Bbb R^n)}
\end{equation}
for any admissible pairs $(q_i,r_i)$, $1\leq i\leq m$.  Here, $p'$ denotes the conjugate exponent to $p$, that is, $\tfrac 1p + \tfrac 1{p'}=1$.
\end{lem}

 The reader must have in mind that wherever
 in this paper we restrict the functions in frequency we do it in a smooth way using the Littlewood-Paley projections. To address 
the frequency localization in a more precise way we need 
some Littlewood-Paley theory.  Specifically, let $\varphi(\xi)$ be a smooth bump supported in $|\xi| \leq 2$
and equalling one on $|\xi| \leq 1$.  For each dyadic number $N \in 2^\Bbb Z$ we define the Littlewood-Paley operators

$$\widehat{P_{\leq N}f}(\xi):=  \varphi(\xi/N)\hat f (\xi),$$

$$\widehat{P_{> N}f}(\xi):=  [1-\varphi(\xi/N)]\hat f (\xi),$$

$$\widehat{P_N f}(\xi):=  [\varphi(\xi/N) - \varphi (2 \xi /N)] \hat f (\xi).$$
\\
Similarly, we can define $P_{<N}$, $P_{\geq N}$, and $P_{M < \cdot \leq N} := P_{\leq N} - P_{\leq M}$, whenever $M$ and $N$ are dyadic
numbers.  We will frequently write $f_{\leq N}$ for $P_{\leq N} f$ and similarly for the other operators. Using the Littlewood-Paley
decomposition we write, at least formally, $u=\sum_{N}P_{N}u$. We can write
 $u=\sum u_{N}$ and obtain bounds on each piece separately or by examining the
 interactions of the several pieces. We can recover information for the original function $u$ by applying the
 Cauchy-Schwarz inequality and using the Littlewood-Paley Theorem \cite{es} or the cheap Littlewood-Paley inequality
$$\|P_{N}u\|_{L^p} \lesssim \|u\|_{L^{p}}$$
\\
for any $1\leq p \leq \infty$. Since this process is fairly standard
 we will often omit the details of the argument throughout the paper.
\\
\\
We also recall the following
standard Bernstein and Sobolev type inequalities. The proofs can be found in \cite{tt}.

\begin{lem}\label{bernstein}
For any $1\le p\le q\le\infty$ and $s>0$, we have
\\
\begin{align*}
\|P_{\geq N} f\|_{L^p_x} &\lesssim N^{-s} \| |\nabla|^s P_{\geq N} f \|_{L^p_x}\\
\| |\nabla|^s  P_{\leq N} f\|_{L^p_x} &\lesssim N^{s} \| P_{\leq N} f\|_{L^p_x}\\
\| |\nabla|^{\pm s} P_N f\|_{L^p_x} &\sim N^{\pm s} \| P_N f \|_{L^p_x}\\
\|P_{\leq N} f\|_{L^q_x} &\lesssim N^{\frac{1}{p}-\frac{1}{q}} \|P_{\leq N} f\|_{L^p_x}\\
\|P_N f\|_{L^q_x} &\lesssim N^{\frac{1}{p}-\frac{1}{q}} \| P_N f\|_{L^p_x}.
\end{align*}
\end{lem}

\vspace{0.4cm}

For $N>1$, we define the Fourier multiplier $I:=I_N$
$$
\widehat{I_N u}(\xi):=m_N(\xi)\hat u(\xi),
$$
where $m_N$ is a smooth radial decreasing function such that
\\
$$
m_N(\xi)=\left\{
\begin{array}{cc}
1, & \text{if} \quad  |\xi|\le N\\
\bigl(\frac{|\xi|}{N}\bigr)^{s-1}, & \text{if} \quad |\xi|\ge
2N.
\end{array}
\right.
$$
\\
Thus, $I$ is the identity operator on frequencies $|\xi|\le N$ and behaves like a fractional integral operator of order $1-s$ on higher
frequencies. In particular, $I$ maps $H^s_x$ to $H_x^1$.  We collect the basic properties of the $I$ operator into the following

\begin{lem}\label{basic property}
Let $1<p<\infty$ and $0\leq\sigma\le s<1$.  Then,
\\
\begin{align}
\|If\|_p&\lesssim \|f\|_p \label{i1}\\
\||\nabla|^\sigma P_{> N}f\|_p&\lesssim N^{\sigma-1}\|\nabla I f\|_p \label{i2}\\
\|f\|_{H^s_x}\lesssim \|If\|_{H^1_x}&\lesssim
N^{1-s}\|f\|_{H^s_x}.\label{i3}
\end{align}
\end{lem}
\begin{proof}
The estimate \eqref{i1} is a direct consequence of H\"ormander's multiplier
theorem.

To prove \eqref{i2}, we write
$$
\||\nabla|^\sigma P_{> N} f\|_p=\|P_{> N}|\nabla |^\sigma(\nabla
I)^{-1}\nabla I f\|_p.
$$
\\
The claim follows again from H\"ormander's multiplier theorem.
\\
\\
Now we turn to \eqref{i3}.  By the definition of the operator $I$
and \eqref{i2},
\\
\begin{align*}
\|f\|_{H^s_x}&\lesssim \|P_{\le N} f\|_{H^s_x}+\|P_{>N}f\|_2+\||\nabla|^s P_{>N} f\|_2\\
&\lesssim \|P_{\le N} I f\|_{H_x^1}+N^{-1}\|\nabla I f\|_2+N^{s-1}\|\nabla I f\|_2\lesssim \|If\|_{H^1_x}.
\end{align*}
On the other hand, since the operator $I$ commutes with $\langle
\nabla\rangle^s$,
\\
\begin{align*}
\|If\|_{H_x^1} =\|\langle\nabla\rangle^{1-s}I\langle\nabla\rangle^s
f\|_2 \lesssim N^{1-s}\|\langle\nabla\rangle^sf\|_2 \lesssim
N^{1-s}\|f\|_{H^s_x},
\end{align*}
\\
which proves the last inequality in \eqref{i3}.  Note that a similar
argument also yields
\begin{align}\label{i4}
\|If\|_{\dot H^1_x}&\lesssim N^{1-s}\|f\|_{\dot H^s_x}.
\end{align}
\end{proof}
\section{Correlation estimates in all dimensions.}
We consider solutions of the following equation
\begin{equation}\label{mor12}
iu_{t}+\Delta u=|u|^{p-1}u, ~ (x,t) \in \R^n \times [0,T].
\end{equation} 
We want to obtain a monotonicity formula that takes advantage of the momentum conservation law of the equation
$$\vec{p}(t)=\Im(\overline{u}(x,t)\nabla u(x,t))=\vec{p}(0).$$
We define the Morawetz action
$$M_a(t)= 2 \int_{\Bbb R^n}\nabla a(x)
\cdot \Im(\overline{u}(x)\nabla u(x))dx$$
where $a: \mathbb R^n \rightarrow \Bbb R$, a convex and locally
integrable function of polynomial growth. By differentiating $M_{a}(t)$ with respect to time and using the conservation laws of
 the equation we will obtain a priori estimates for solutions of \eqref{mor12}. To accomplish that we make a clever choice of the weight function
 $a(x)$. We note that in all of the cases that we will consider we pick $a(x)=f(|x|)$ where $f: \Bbb R \rightarrow \Bbb R$ is a convex
 function with the property that $f^{\prime}(x) \geq 0$ for $x \geq 0$. Then a simple calculation shows that the second derivative matrix of $a(x)$ is 
given by 
$$\partial_{j}\partial_{k}a(x)=f^{''}(|x|)\frac{x_jx_k}{|x|^2}+\frac{f^{'}(|x|)}{|x|}(\delta_{kj}-\frac{x_jx_k}{|x|^2})$$
\\
But then the quadratic form $\langle y_jy_{k}\ |\ \partial_{j}\partial_{k}a(x)\rangle$ is positive definite
since
\\
$$\langle y_jy_{k}\ |\ \partial_{j}\partial_{k}a(x)\rangle=f^{''}(|x|)\frac{(x\cdot y)^2}{|x|^2}+f^{'}(|x|)(|y|^2-\frac{(x\cdot y)^2}{|x|^2}) \geq 0$$
\\
by the Cauchy-Schwarz inequality 
$$|x \cdot y| \leq |x||y|.$$
As a final comment for the careful reader we note that in all our arguments we will assume smooth solutions. This will simplify
 the calculations and will enable us to justify the steps in the subsequent proofs. The local well-posedness theory
 and the perturbation theory \cite{tc} that has been established for this problem can be then applied to approximate
 the $H^{s}$ solutions by smooth solutions and conclude the proofs. For most of the calculations in this section the reader can consult \cite{ckstt5}, \cite{tt}.  

The equation satisfies the following local conservation laws.
\\
\\
Local mass conservation
$$\partial_{t}\rho+\partial_{j}p_{j}=0$$
and local momentum conservation
$$\partial_{t}p_k+\partial_{k}\left(\sigma_{jk}+\delta_{kj}\left(-\Delta \rho+2^{\frac{p+1}{2}}\frac{p-1}{p+1}\rho^{\frac{p+1}{2}}\right)\right)=0$$
where 
$$\rho=\frac{1}{2}|u|^2$$
is the mass density,
$$p_{j}=\Im(\bar{u}\partial_{j}u)$$
is the momentum density,
and 
$$\sigma_{jk}=\frac{1}{\rho}(p_jp_k+\partial_j \rho \partial_k \rho)$$
is a stress tensor. Using the identity 
$$\Re(z_1\bar{z}_2)=\Im z_1 \Im z_2+\Re z_1 \Re z_2$$
we can write
$$\sigma_{jk}=\frac{1}{\rho}(p_jp_k+\partial_j \rho \partial_k \rho)=2\Re(\partial_{k}u\partial_j\bar{u}).$$ 
In what follows we will use both definitions of $\sigma_{jk}$ according to what we find more appropriate with the situation
 at hand. Note that integration of the first equation leads to mass conservation while integration of the second leads to momentum conservation.
We are ready to prove the {\it{generalized virial identity}} \cite{ls}.
\begin{prop}\label{wmor}
If $a$ is convex and $u$ is 
a smooth solution to equation \eqref{mor12} on $[0,T]\times \mathbb R^{n}$. 
Then, the following inequality holds:
\begin{equation}\label{Mor}
\int_0^T\int_{\Bbb R^n}(-\Delta \Delta a)|u(x,t)|^2dxdt \lesssim 
\sup_{[0,T]}|M_{a}(t)|,
\end{equation}
where $M_{a}(t)$ is the Morawetz action and is given by
\begin{equation}
\label{Mat}M_a(t)= 2 \int_{\Bbb R^n}\nabla a(x)
\cdot \Im(\overline{u}(x)\nabla u(x))dx.
\end{equation}
\end{prop}
\begin{proof}
We can write the Morawetz action as
$$M_{a}(t)=2\int_{\Bbb R^n} (\partial_{j}a)p_jdx.$$
Then
$$\partial_{t}M_a(t)=2\int_{\Bbb R^n} (\partial_{j}a)\partial_{t}p_{j}dx=2\int_{\Bbb R^n} 
\partial_{j}a\left( -\partial_{k}\left(\sigma_{jk}+\delta_{kj}\left(-\Delta \rho+2^{\frac{p+1}{2}}\frac{p-1}{p+1}\rho^{\frac{p+1}{2}}\right)\right)\right)dx=$$
\\
$$2\int_{\Bbb R^n}(\partial_{j}\partial_{k}a)\sigma_{jk}dx-2\int_{\Bbb R^n} 
\partial_{j}a\partial_{j}\left(-\Delta \rho+2^{\frac{p+1}{2}}\frac{p-1}{p+1}\rho^{\frac{p+1}{2}}\right)dx=$$
\\
$$4\int_{\Bbb R^n}(\partial_{j}\partial_{k}a)\Re(\partial_{k}u\partial_j\bar{u})dx+2\int_{\Bbb R^n}\Delta a 
\left( -\Delta \rho+2^{\frac{p+1}{2}}\frac{p-1}{p+1}\rho^{\frac{p+1}{2}}\right)dx=$$
\\
$$4\int_{\Bbb R^n}(\partial_{j}\partial_{k}a)\Re(\partial_{k}u\partial_j\bar{u})dx+2\int_{\Bbb R^n}(-\Delta \Delta a)\rho dx+
2^{\frac{p+3}{2}}\frac{p-1}{p+1}\int_{\Bbb R^n}(\Delta a)\rho^{\frac{p+1}{2}}dx=$$
\\
$$4\int_{\Bbb R^n}(\partial_{j}\partial_{k}a)\Re(\partial_{k}u\partial_j\bar{u})dx+\int_{\Bbb R^n}(-\Delta \Delta a) |u|^{2}dx+\frac{2(p-1)}{p+1}
\int_{\Bbb R^n}(\Delta a)|u|^{p+1}dx.$$
\\
To prove this identity we used the local conservation of momentum law, integration by parts and the definitions of $\rho$ and $\sigma_{jk}$. 
But since $a$ is convex we have that
\\ 
$$4(\partial_{j}\partial_{k}a)\Re \left( \partial_{j}\bar{u}\partial_{k}u\right) \geq 0$$
and the trace of the Hessian of $\partial_{j}\partial_{k}a$ which is $\Delta a$ is positive. Thus,
\\
$$\int_{\Bbb R^n}(-\Delta \Delta a) |u|^{2}dx \leq \partial_{t}M_a(t)$$
\\
and by the fundamental theorem of calculus we have that
\begin{equation}
\int_0^T\int_{\Bbb R^n}(-\Delta \Delta a)|u(x,t)|^2dxdt \lesssim 
\sup_{[0,T]}|M_{a}(t)|. \label{mora}
\end{equation}
\end{proof}
\subsection{Interaction Morawetz inequality in dimension $n\geq 3$.}
Using the approach above we can derive correlation estimates that are very useful in studying the global well-posedness and the scattering
 properties of nonlinear dispersive partial differential equations. For clarity in this subsection we reproduce some calculations that have 
appeared in \cite{cgt}. Let $u_{i}$, $F_{i}$ 
be solutions to 
\begin{equation}\label{F}
iu_{t}+\Delta u=F(u)
\end{equation} 
in $n_{i}-$spatial dimensions. Define the tensor product 
$u:=(u_{1}\otimes u_{2})(t,x)$ for $x$ in 
\\
$$\mathbb R^{n_{1}+n_{2}}=\{(x_{1},x_{2}): x_{1} \in \mathbb R^{n_{1}}, x_{2} \in \mathbb R^{n_{2}}\}$$ 
by the formula 
$$(u_{1}\otimes u_{2})(t,x)=u_{1}(x_{1},t)u_{2}(x_{2},t).$$
\\
We abbreviate $u(x_i)$ by $u_i$ and note that if $u_1$ solves (\ref{F}) with forcing term $F_1$ and  $u_2$ solves (\ref{F}) 
with forcing term $F_2$, 
then $u_1 \otimes u_2$ solves (\ref{F}) with forcing term $F=F_{1}\otimes u_2+F_{2}\otimes u_1$. 
We have that
$$\rho=\frac{1}{2}|u(x)|^2=2\rho_{1}\rho_{2},$$
\\
$$p_k=\Im(\overline{u_{1}u_{2}}\partial_{j}(u_{1}u_{2}))=\left(p_{k}(u_1)\otimes |u_2|^2,p_{k}(u_2)\otimes |u_1|^2\right),$$
\\
$$\sigma_{jk}=2\Re(\partial_j(u_1u_2)\partial_k(\overline{u_1u_2}))=2\left(\sigma_{jk}(u_1)\otimes |u_2|^2+\sigma_{jk}(u_2)\otimes |u_1|^2\right),$$
\\
where $\rho_{i}=\frac{1}{2}|u_i|^2,\ \ \ i=1,2,$ and similarly for $p_{k}(u_i)$, $\sigma_{jk}(u_i)$.
Then the local conservation laws can be written in the following way
$$\partial_{t}\rho+\partial_{j}p_{j}=0$$
\\
$$\partial_{t}p_k+\partial_{k}\left(\sigma_{jk}-\delta_{kj}\Delta \rho+G\right)=0$$
\\
where
$$G=2^{\frac{p+1}{2}}\frac{p-1}{p+1}\left( G_1 \otimes |u_2|^2+G_2\otimes |u_1|^2\right)\geq 0$$
and $G_{i}=G(u_{i})=\rho_{i}^{\frac{p+1}{2}}$. Of course in this setting $\nabla=(\nabla_{x_1},\nabla_{x_2})$ and $\Delta=\Delta_{x_1}+\Delta_{x_2}$.
\\
\\
If we now apply Proposition
\ref{wmor} for the tensor product of the two solutions we obtain for a convex functions $a$ that
\\
\begin{equation}\label{interact}
\int_0^T\int_{\Bbb R^{n_{1}} \otimes \Bbb R^{n_{2}}}(-\Delta \Delta a)|u_1\otimes u_2|^2(x,t)dxdt \lesssim 
\sup_{[0,T]}|M_{a}^{\otimes_{2}}(t)|
\end{equation}
where again $\Delta= \Delta_{x_1}+\Delta_{x_2}$ the Laplacian in $\R^{n_1 +
  n_2}$ and $M_{a}^{\otimes_{2}}(t)$ is the Morawetz
 action that corresponds to $u_1\otimes u_2$ and thus
\\
$$M_{a}^{\otimes_{2}}(t)= 2 \int_{\Bbb R^{n_{1}}\otimes \Bbb R^{n_{2}}}\nabla a(x)
\cdot \Im\left(\overline{u_1\otimes u_2}(x)\nabla (u_1\otimes u_2(x))\right )dx=$$
\\
$$M_{a}(u_{1}(t))\otimes |u_2|^2+M_{a}(u_{2}(t))\otimes |u_1|^2.$$
\\
Now we pick $a(x)=a(x_1,x_2)=|x_1-x_2|$ where $(x_{1},x_2) \in \Bbb R^{n}\times \Bbb R^{n}$. 
Then an easy calculation shows that 
\[-\Delta \Delta a(x_1,x_2)=\left\{\begin{array}{ll}
C_1\delta(x_1-x_2)& \mbox{if $n=3$}\\
\frac{C_{2}}{|x_1-x_2|^3} & \mbox{if $n\geq 4$}
\end{array}
\right.\]
\\
where $C_1,C_2$ are constants.
 Applying equation \eqref{interact} with this choice of $a$ and choosing $u_{1}=u_{2}$ we get that in the case that $n=3$
\\
$$\int_{0}^{T}\int_{\Bbb R^3}|u(x,t)|^{4}dx \lesssim \sup_{[0,T]}|M_{a}^{\otimes_{2}}(t)|$$
and in the case that $n \geq 4$
\\
$$\int_{0}^{T}\int_{\Bbb R^n\otimes \Bbb R^n}|u(x_2,t)|^2 \frac{|u(x_1,t)|^2}{|x_1-x_2|^3}dx_1dx_2dt\lesssim \sup_{[0,T]}|M_{a}^{\otimes_{2}}(t)|.$$
But
$$\int_{0}^{T}\int_{\Bbb R^n\otimes \Bbb R^n}|u(x_2,t)|^2 \frac{|u(x_1,t)|^2}{|x_1-x_2|^3}dx_1dx_2dt
=\int_{0}^{T}\int_{\Bbb R^n}(|u|^{2}\star \frac{1}{|\cdot|^3})(x)|u(x)|^2dxdt.$$
\\
Now we define for $n \geq 4$ the integral operator
$$D^{-(n-3)}f(x):=\int_{\Bbb R^n}\frac{u(y)}{|x-y|^3}dy$$
where $D$ stands for the derivative. This is indeed defined since for $n \geq 4$ the distributional Fourier transform of $|x|^{-3}$ is given by
$$\widehat{|\cdot|^{-3}}(\xi)=|\xi|^{-(n-3)}.$$
By applying Plancherel's Theorem and distribute the derivatives we obtain that
\\
$$\int_{0}^{T}\int_{\Bbb R^n\otimes \Bbb R^n}|u(x_2,t)|^2 \frac{|u(x_1,t)|^2}{|x_1-x_2|^3}dx_1dx_2dt=\int_{0}^{T}\int_{\Bbb R^n}|D^{-\frac{n-3}{2}}(|u(x)|^2)|^2dxdt.$$
Thus we obtain that
\\
$$\int_{0}^{T}\int_{\Bbb R^n}|D^{-\frac{n-3}{2}}(|u(x)|^2)|^2dxdt \lesssim \sup_{[0,T]}|M_{a}^{\otimes_{2}}(t)|.$$
For the sake of simplicity we combine the two estimates for $n\geq 3$ pretending that $D^{0}={\bf 1}$ into
$$\|D^{-\frac{n-3}{2}}(|u(x)|^2)\|_{L_{t}^{2}L_{x}^{2}}^2 \lesssim \sup_{[0,T]}|M_{a}^{\otimes_{2}}(t)|.$$
It can be shown using Hardy's inequality (for details see \cite{ckstt4}) that for $n \geq 3$
$$\sup_{[0,T]}|M_{a}(t)| \lesssim \sup_{[0,T]}\|u(t)\|_{\dot{H}^{\frac{1}{2}}}^2$$
Since we have that
$$M_{a}^{\otimes_{2}}(t)=M_{a}(u_{1}(t))\otimes |u_2|^2+M_{a}(u_{2}(t))\otimes |u_1|^2$$
we obtain
\begin{equation}\label{nmor}
\|D^{-\frac{n-3}{2}}(|u(x)|^2)\|_{L_{t}^{2}L_{x}^{2}}^2\lesssim \sup_{[0,T]} \|u(t)\|_{\dot{H}^{\frac{1}{2}}}^2\|u(t)\|_{L^{2}}^2
\end{equation}
which is the interaction Morawetz estimates that appears in \cite{ckstt4} and in \cite{mv}.
\begin{rem}
The above method breaks down for $n<3$ since the distribution $-\Delta \Delta (|x|)$ is not positive anymore.
\end{rem} 
\subsection{Interaction Morawetz inequality in two dimensions.}
In two dimensions in \cite{cgt}, we 
follow an alternative approach. In that case $(x_1,x_2)\in \mathbb R^2 \times \mathbb R^2$. The idea is again to consider the tensor product of
 two solutions but with a different weight function. We couldn't prove that $-\Delta \Delta a(x)$ is positive. Instead we obtained a difference
 of two positive functions and we balanced the two terms by picking the constants in an appropriate way. The details were as follows.
\\
\\
Let $f:[0,\infty) \rightarrow [0,\infty)$ be such that
\[f(x):= \left\{\begin{array}{ll}
\frac{1}{2M}x^{2}(1-\log{\frac{x}{M}})& \mbox{if $|x|<\frac{M}{\sqrt{e}}$}\\
100x & \mbox{if $|x|>M$}\\
smooth\ and\ convex\ for\ all\ x & \mbox{}
\end{array}
\right.\]
\\
and $M$ is a large parameter that we will choose later. It is obvious that the functions $\frac{1}{2M}x^{2}(1-\log{\frac{x}{M}})$
 and $100x$ are convex in their domain, and the graph of either function lies strictly above the tangent lines of the other.
 Thus one can construct a function with the above properties. Note also that for $x \geq 0$ we have that 
$f^{\prime}(x) \geq 0$. If we apply Proposition \ref{wmor} with the weight 
$a(x_{1},x_{2})=f(|x_1-x_2|)$ and tensoring again two functions we conclude that
$$\int_0^T\int_{\mathbb R^2 \times \mathbb R^2}(-\Delta \Delta a(x_1,x_2))|u(x_{1},t)|^2|u(x_2,t)|^2dx_1dx_2dt \lesssim 2
\sup_{[0,T]}|M_{a}^{\otimes_{2}}(t)|$$
\\
But for $|x_1-x_2|<\frac{M}{\sqrt{e}}$ we have that $\Delta a(x_1,x_2)= \frac{2}{M}\log (\frac{M}{|x_1-x_2|})$ and thus 
$$-\Delta \Delta a(x_1,x_2)=\frac{4\pi}{M}\delta_{\{x_{1}=x_{2}\}}.$$ 
\\
On the other hand for $|x_1-x_2|>M$ we have that
 $$-\Delta \Delta a(x_1,x_2)= O(\frac{1}{|x_1-x_2|^3})=O(\frac{1}{M^{3}}).$$
\\
We have a similar bound in the region in between just because
$a(x_1,x_2)$ is smooth, so all in all, we have
$$-\Delta \Delta a(x_1,x_2)=\frac{4\pi}{M}\delta_{\{x_{1}=x_{2}\}}+O(\frac{1}{M^{3}}).$$
Thus 
$$\int_0^T\int_{\mathbb R^2 \times \mathbb R^2}(-\Delta \Delta a(x_1,x_2))|u(x_{1},t)|^2|u(x_2,t)|^2dx_1dx_2dt=
O(\frac{1}{M})\int_0^T\int_{\mathbb R^2}|u(x,t)|^4dxdt+$$
$$O(\frac{1}{M^{3}})\int_0^T\int_{\mathbb R^2 \times \mathbb R^2}
|u(x_1,t)|^2|u(x_2,t)|^2dx_1dx_2dt.$$
\\
By Fubini's Theorem
\begin{equation}
\label{crudestep}
\frac{C}{M^{3}}\int_0^T\int_{\mathbb R^2 \times \mathbb R^2}
|u(x_1,t)|^2|u(x_2,t)|^2dx_1dx_2dt \lesssim
\frac{CT}{M^{3}}\|u\|_{L_{t}^{\infty}L_{x}^{2}}^4.
\end{equation}
\\
On the other hand by the analogue of Hardy's inequality in 2d we have
\\
$$\sup_{[0,T]}|M_{a}^{\otimes_{2}}(t)|\lesssim \sup_{[0,T]}\|u\|_{L_{t}^{\infty}L_{x}^{2}}^2\|u\|_{L_{t}^{\infty}\dot{H}_{x}^{\frac{1}{2}}}^2.$$
\\
Thus by applying Proposition \ref{wmor}
$$\frac{1}{M}\int_0^T\int_{\mathbb R^2}|u(x,t)|^4dxdt\lesssim 
\sup_{[0,T]}\|u\|_{L_{t}^{\infty}L_{x}^{2}}^2\|u\|_{L_{t}^{\infty}\dot{H}_{x}^{\frac{1}{2}}}^2+\frac{T}{M^{3}}\|u\|_{L_{t}^{\infty}L_{x}^{2}}^4.$$
\\
Multiplying the above equation by $M$ and balancing the two terms on the right hand side 
by picking 
$$M \sim T^{\frac{1}{3}}\left(\frac{\|u\|_{L_{t}^{\infty}L_{x}^{2}}}{\|u\|_{L_{t}^{\infty}\dot{H}_{x}^{\frac{1}{2}}}}\right)^{\frac{2}{3}}$$ 
\\
we get a better estimate than was obtained in \cite{fg}
\\
$$\|u\|_{L_{t\in [0,T]}^{4}L_{x}^{4}}^4 \lesssim T^{\frac{1}{3}}\|u\|_{L_{t}^{\infty}L_{x}^{2}}^{\frac{8}{3}}\|u\|_{L_{t}^{\infty}\dot{H}_{x}^{\frac{1}{2}}}^{\frac{4}{3}}.$$
\subsection{A new correlation estimate in two dimensions.  Proof of Theorem 1.}
We can refine the tensor product approach of the previous subsection and find a convex weight function $a(x)$ such that the distribution 
$-\Delta \Delta a(x)$ is positive. Notice that so far we have used $a(r=|x|)$ such that $a(r) \sim r^{2}\log{\frac{1}{r}}$ for $r \sim 0$ and $a(r)\sim r$
 for large values of $r$. In between we didn't provide an explicit formula but used only the quantitative properties of the function. We would like
 to follow this path one more time and implicitly define a radial function $a: \Bbb R^2 \rightarrow \Bbb R$ such that
$$\Delta a(r)=\int_{r}^{\infty}s\log (\frac{s}{r})w_{r_{0}}(s)ds$$
where 
\[w_{r_{0}}(s):= \left\{\begin{array}{ll}
\frac{1}{s^3}& \mbox{if $s\geq r_{0}$}\\
0 & \mbox{otherwise}
\end{array}
\right.\]
and $r_{0}>0$ and small. Notice that for $r<r_{0}$ an explicit calculation shows that
\\
$$\Delta a(r)=\int_{r}^{\infty}s\log (\frac{s}{r})w_{r_0}(s)ds=
\int_{r_0}^{\infty}s\log (\frac{s}{r})w_{r_0}(s)ds=$$
$$\int_{r_0}^{\infty}s\log (\frac{s}{r})\frac{1}{s^3}ds=\frac{1}{r_{0}}\left(1+\log{\frac{r_{0}}{r}}\right).$$
\\
Now if we solve the equation
$$a_{rr}(r)+\frac{1}{r}a_{r}(r)=\Delta a(r)$$
with the initial conditions $a(0)=0$ and $a_{r}(0)=0$ we obtain that for $r<r_{0}$
\\
$$a(r)=\frac{r^2}{2r_{0}}\left(1+\frac{1}{2}\log{\frac{r_{0}}{r}}\right).$$
\\
Moreover for $r \geq r_{0}$ we have 
\\
$$\Delta a(r)=\int_{r}^{\infty}s\log (\frac{s}{r})w_{r_{0}}(s)ds=\int_{r}^{\infty}s\log (\frac{s}{r})\frac{1}{s^3}ds=\frac{1}{r}.$$
Solving again
$$a_{rr}(r)+\frac{1}{r}a_{r}(r)=\Delta a(r)$$ 
we obtain that for $r \geq r_{0}$
$$a(r)=r.$$
Thus the weight function $a$ is a function that we have already seen. In addition by the definition of $w(x)$ and $a(x)$ we have that
$$\Delta a \geq 0$$
and
$$\int_{\Bbb R^2}w_{r_{0}}(|\vec{x}|)dx=\frac{2\pi}{r_{0}} \ or \ \int_{0}^{\infty} sw_{r_{0}}(s)ds=\frac{1}{r_{0}}.$$
\\
$\Delta a$ can be rewritten as
$$\Delta a=\int_{0}^{\infty}sw_{r_{0}}(s)\log(\frac{s}{r})ds-\int_{0}^{r}sw_{r_{0}}(s)\log(\frac{s}{r})ds=$$
\\
$$-\frac{1}{r_{0}}\log(r)+\int_{0}^{\infty}sw_{r_0}(s)\log(s)ds+\int_{0}^{r}sw_{r_0}(s)\log(\frac{r}{s})ds.$$
\\
By setting $\log{C}=r_{0}\int_{0}^{\infty}sw_{r_{0}}(s)\log(s)ds$ we can write 
\\
$$\Delta a=\frac{1}{r_{0}}\log(\frac{C}{r})+p(r)$$
where 
$$p(r)=\int_{0}^{r}w_{r_{0}}(s)s\log(\frac{s}{r})ds.$$
It is immediate that the Laplacian of the radial function $p$ is $w_{r_{0}}(r)$ 
as an explicit calculation shows using the fact that $$\Delta p=p_{rr}+\frac{1}{r}p_{r}$$ 
Thus $\Delta p =w_{r_{0}}$ and
$$-\Delta \Delta a(|x|)=\frac{2\pi}{r_{0}}\delta(|x|)-w_{r_{0}}(|x|)$$
\\
\\
We want to apply Proposition \ref{mor12} with $a(\vec{x_1},\vec{x_2})=a(|\vec{x_1}-\vec{x_2}|)$ 
to a tensor product of two functions. We need to prove that $a(r)$ is convex and as we have already mentioned this will be
 immediate if we establish that $a_{rr}\geq 0$ and $a_{r}\geq 0$. Assuming this is true we obtain 
\\
\begin{equation}\label{newmor}
\int_{0}^{T}\frac{2\pi}{r_{0}}\int_{\Bbb R^2} |u(\vec{x})|^4d\vec{x}dt-\int_{0}^{T}\int_{\Bbb R^2 \times \Bbb R^2}w_{r_{0}}(|\vec{x_1}-\vec{x_2}|) 
|u(\vec{x_1})|^2|u(\vec{x_2})|^2 d\vec{x_1}d\vec{x_2}dt\lesssim 
\sup_{[0,T]}|M_{a}^{\otimes_{2}}(t)|.
\end{equation}
\\
The left hand side can be rewritten as
$$\int_{0}^{T}\frac{2\pi}{r_{0}}\int_{\Bbb R^2} |u(x)|^4dxdt-\int_{0}^{T}\int_{\Bbb R^2 \times \Bbb R^2}w_{r_{0}}(|x_1-x_2|) 
|u(x_1)|^2|u(x_2)|^2 dx_1dx_2dt=$$
\\
$$\frac{1}{2}\int_{0}^{T}\frac{2\pi}{r_{0}}\int_{\Bbb R^2} |u(x_{1})|^4dx_{1}dt+
\frac{1}{2}\int_{0}^{T}\frac{2\pi}{r_{0}}\int_{\Bbb R^2} |u(x_{2})|^4dx_{2}dt-$$
\\
$$\int_{0}^{T}\int_{\Bbb R^2 \times \Bbb R^2}w_{r_{0}}(|x_1-x_2|) 
|u(x_1)|^2|u(x_2)|^2 dx_1dx_2dt.$$
\\
Taking into account that
$$\int_{\Bbb R^2}w_{r_{0}}(|x_1-x_2|)dx_1=\frac{2\pi}{r_{0}}$$
we can rewrite \eqref{newmor} as
\\
$$\int_{0}^{T}\int_{\Bbb R^2 \times \Bbb R^2}\{|u(t,x_1)|^2-|u(t,x_2)|^2\}^2w_{r_{0}}(|x_1-x_2|) dx_1dx_2dt \lesssim 
\sup_{[0,T]}|M_{a}^{\otimes_{2}}(t)|.$$
\\
Since this last estimate is true for every $r_{0}>0$, by taking the limit as $r_{0}\rightarrow 0$ we obtain
\\
$$\int_{0}^{T}\int_{\Bbb R^2 \times \Bbb R^2}\frac{\{|u(t,x_1)|^2-|u(t,x_2)|^2\}^2}{|x_1-x_2|^3} dx_1dx_2dt \lesssim 
\sup_{[0,T]}|M_{a}^{\otimes_{2}}(t)|.$$
But
$$\int_{\Bbb R^2 \times \Bbb R^2}\frac{\{|u(t,x_1)|^2-|u(t,x_2)|^2\}^2}{|x_1-x_2|^3} dx_1dx_2=\||u|^2\|_{\dot{H}^{\frac{1}{2}}}^2$$
\\
see for example \cite{bl} exercise 7 on page 162. Thus we get 
\\
$$\| D^{1/2} |u|^2 \|_{L^2_t L^2_x}^2\lesssim \sup_{[0,T]}|M_{a}^{\otimes_{2}}(t)|.$$
\\
If $\vec{\nabla}a=\frac{\vec{x}}{|\vec{x}|}a_r$ is bounded, we can estimate
 $M_{a}^{\otimes_{2}}(t)$ as before and obtain the new a priori correlation estimate for solutions of (\ref{mor12})
\\
$$\| D^{1/2} |u|^2 \|_{L^2_t L^2_x}^2 \lesssim \| u \|_{L^\infty_t L^2_x}^2  \| u \|_{L^\infty_t {\dot{H}}^{1}_x}^2.$$
\\
Thus it remains to establish that $a(r)$ is convex and that $a_{r}(r)$ is bounded. 
For $r$ near zero we have that $a(r)\sim r^2\log(\frac{1}{r})$ and thus $a_r$ is bounded for
 small values of $r$. 
In particular $a_r(0)=0$. Using this as an initial condition we can solve in terms of $a_r$ the equation $a_{rr}+\frac{1}{r}a_r=\Delta a$  
and obtain
$$a_r(r)=\frac{1}{r}\int_{0}^{r}s(\Delta a)(s)ds \geq 0.$$
Thus $a_{r} \geq 0$. We will shortly show that $a_{rr} \geq 0$ for any $r>0$ and thus $a_r(r)$ is a positive increasing function. 
Because of this it is enough to consider
 the values of $a_r$ for large values of $r$. Recall that 
$$\Delta a(r)=\int_{r}^{\infty}s\log (\frac{s}{r})w_{r_{0}}(s)ds$$
and that for $s \geq r_{0}$, $w_{r_{0}}(s)=\frac{1}{s^3}$. Thus
$$\Delta a(r)=\int_{r}^{\infty}\frac{1}{s^2}\log (\frac{s}{r})ds=\frac{1}{r}.$$
Since for $s \geq r_{0}$ we have $\Delta a(s)=\frac{1}{s}$, for $r \geq s \geq r_{0}$ we obtain
$$a_r(r)=\frac{1}{r}\int_0^rds=1$$
and $a_r$ is bounded. It remains to show that $a_{rr} \geq 0$ for $r \geq 0$. To this end notice that
\\
$$a_{rr}(r)=\Delta a(r)-\frac{1}{r}a_{r}(r)=\Delta a(r)-\frac{1}{r^2}\int_{0}^{r}s(\Delta a)(s)ds=\frac{q(r)}{r^2}$$
where 
$$q(r)=\int_{0}^{r}[2\Delta a(r)-\Delta a(s)]sds.$$
It remains to show that $q(r) \geq 0$. Since $q(0)=0$, it is enough to show that $q_{r}(r) \geq 0$. An elementary calculation shows that
\\
$$q_{r}(r)=r[\Delta a(r)+r(\Delta a)^{\prime}(r)].$$
Thus it remains to show that 
$$\Delta a(r)+r(\Delta a)^{\prime}(r)\geq 0.$$
Recalling one more time that
$$\Delta a(r)=\int_{r}^{\infty}s\log (\frac{s}{r})w_{s}(s)ds.$$
If we differentiate with respect to $r$ we obtain
$$\Delta a)^{\prime}(r)=-\frac{1}{r}\int_{r}^{\infty}sw_{r_{0}}(s)ds.$$
Thus
$$\Delta a(r)+r(\Delta a)^{\prime}(r)=\int_{r}^{\infty}sw_{r_{0}}(s)\log{\frac{s}{re}}ds$$
A calculation shows that
\[\Delta a(r)+r(\Delta a)^{\prime}(r)= \left\{\begin{array}{ll}
\frac{1}{r_{0}}\log{\frac{r_{0}}{r}}& \mbox{if $r<r_{0}$}\\
0 & \mbox{$if r \geq r_{0}$}
\end{array}
\right.\]
\\
and thus we are done. 
%Note that all the above calculations would be done explicitly by using the formulas for $a(r)$ that we recovered
%\[a(r)= \left\{\begin{array}{ll}
%\frac{r^2}{2r_{0}}\left(1+\frac{1}{2}\log{\frac{r_{0}}{r}}\right)& \mbox{if $r<r_{0}$}\\
%r& \mbox{if $r \geq r_{0}$}
%\end{array}
%\right.\]
%In this case
%\[a_{r}(r)= \left\{\begin{array}{ll}
%\frac{r}{2r_{0}}\left(1+\log{\frac{r_{0}}{r}}\right)& \mbox{if $r<r_{0}$}\\
%1& \mbox{if $r \geq r_{0}$}
%\end{array}
%\right.\]
%and
%\[a_{rr}(r)= \left\{\begin{array}{ll}
%\frac{1}{2r_{0}}\left(\frac{1}{2}+\log{\frac{r_{0}}{r}}\right)& \mbox{if $r<r_{0}$}\\
%0& \mbox{if $r \geq r_{0}$}
%\end{array}
%\right.\]
%\\
%and we can draw the same conclusions as before.
\subsection{Commutator vector operators and correlation estimates. An alternative proof of Theorem 1.}
In this subsection we derive correlation estimates by using commutator vector operators acting on the conservation laws of the equation. It turns
 out that this method is more flexible and can also be generalized. Recall that 
$$M_{a}^{\otimes_{2}}(t)= 2 \int_{\Bbb R^{n_{1}}\otimes \Bbb R^{n_{2}}}\nabla a(x)
\cdot \Im\left(\overline{u_1\otimes u_2}(x)\nabla (u_1\otimes u_2(x))\right )dx$$
is the Morawetz action for the tensor product of two solutions $u:=(u_{1}\otimes u_{2})(t,x)$ where $x=(x_1,x_2) \in \Bbb R^n \otimes \Bbb R^n$. If
 we specialize to the case that $u_1=u_2$, $a(x)=|x|$ and $n \geq 2$ and observe that 
$$\partial_{x_1}a(x_1,x_2)=\frac{x_1-x_2}{|x_1-x_2|}=-\frac{x_2-x_1}{|x_1-x_2|}=-\partial_{x_2}a(x_1,x_2)$$ we can view $M_{a}^{\otimes_{2}}(t):=M(t)$
 as
\\ 
\begin{equation}\label{commute}
M(t)=\int_{\Bbb R^n \otimes \Bbb R^n}\frac{x_1-x_2}{|x_1-x_2|}\cdot \{\vec{p}(x_1,t)\rho(x_2,t)-\vec{p}(x_2,t)\rho(x_1,t)\}dx_1dx_2
\end{equation}
\\
where $\rho=\frac{1}{2}|u|^2$ is the mass density and $p_{j}=\Im(\bar{u}\partial_{j}u)$ is the momentum density. Now let's define the integral operator
$$D^{-(n-1)}f(x)=\int_{\Bbb R^n}\frac{1}{|x-y|}f(y)dy$$ 
where $D$ stands for the derivative. This is indeed justified because for $n \geq 2$ the distributional transform of $\frac{1}{|x|}$ is
 $\frac{1}{|\xi|^{n-1}}$. The main observation is that we can write the action term $M(t)$ using a commutator in the following manner,
$$M(t)=\langle [x;D^{-(n-1)}]\rho(t) \  | \ \vec{p}(t)\rangle.$$
This equation follows from an elementary rearrangement of the terms of \eqref{commute}. This suggests that the estimate is derived using the vector
 operator, which we will denote by $\vec{X}$, defined by
$$\vec{X}=[x;D^{-(n-1)}].$$
We change notation and write $x_1:=x$ and $x_2:=y$. 
The crucial property is that the derivatives of this operator $\partial_{j}X^{k}$ form a positive definite operator. Note that in physical space
$$\vec{X}f(x)=\int_{\Bbb R^n}\frac{x-y}{|x-y|}f(y)dy$$ and a calculation shows that
$$\partial_{j}X^{k}=D^{-(n-1)}\delta_{kj}+[x_k ; R_{j}]$$
where $R_{j}$ is the singular integral operator corresponding to the symbol $\frac{\xi_{j}}{|\xi|^{n-1}}$. Thus we have that 
$$R_{j}=\partial_{j}D^{-(n-1)}$$ and acts on function in the following manner
$$R_{j}f(x)=-\int_{\Bbb R^n}\frac{x_j-y_j}{|x-y|^3}f(y)dy.$$
To see how $\partial_{j}X^{k}$ acts on functions we associate a kernel with the commutator $[x_k ; R_{j}]$, let's call it $r_{jk}(x,y)$ and thus
$$[x_k ; R_{j}]f(x)=\int_{\Bbb R^n}r_{jk}(x,y)f(y)dy$$ 
where
$$r_{kj}(x,y)=-\frac{(x_k-y_k)(x_j-y_j)}{|x-y|^3}.$$
Thus 
$$(\partial_{j}X^{k})f(x)=\int_{\Bbb R^n}\eta_{kj}(x,y)f(y)dy$$
where
$$\eta_{kj}(x,y)=\frac{\delta_{kj}|x-y|^2-(x_j-y_j)(x_k-y_k)}{|x-y|^3}$$
\\
and thus the derivatives of the vector operator $\vec{X}$ form a positive definite operator. Note also that the divergence of the vector field
 $\vec{X}$ is given by
\\
$$\nabla \cdot \vec{X}=\partial_{j}X^{j}=nD^{-(n-1)}+[x_j ; R_{j}]=(n-1)D^{-(n-1)}.$$
\\
Now if we differentiate 
$$M(t)=\langle [x;D^{-(n-1)}]\rho(t) \  | \ \vec{p}(t)\rangle=\langle \vec{X}\rho(t) \  | \ \vec{p}(t)\rangle$$ 
we obtain that
\\
\begin{equation}\label{dermor}
\partial_{t}M(t)=\langle \vec{X}\partial_{t}\rho(t) \  | \ \vec{p}(t)\rangle-\langle \vec{X}\cdot \partial_{t}\vec{p}(t) \  | \ \rho(t)\rangle
\end{equation}
where we have used the fact that $\vec{X}$ is an antisymmetric operator. Now recall the local conservation laws
\begin{equation}\label{localmass}
\partial_{t}\rho+\partial_{j}p_{j}=0
\end{equation}
\\
\begin{equation}
\partial_{t}p_k+\partial_{k}\left(\sigma_{jk}+\delta_{kj}\left(-\Delta \rho+2^{\frac{p+1}{2}}\frac{p-1}{p+1}\rho^{\frac{p+1}{2}}\right)\right)=0.
\end{equation}
\\
To simplify the calculations we will treat the cubic nonlinearity ($p=3$) but the method is general and give the same results for the general
 nonlinearity $|u|^{p-1}u$. Thus we have
\\
\begin{equation}\label{localmom}
\partial_{t}p_k+\partial_{k}\left(\sigma_{jk}+\delta_{kj}\left(-\Delta \rho+2\rho^{2}\right)\right)=0.
\end{equation}
\\
Applying the operator to the equation \eqref{localmass} and contracting with $p_{k}$ and similarly applying the operator to equation \eqref{localmom}
 and contracting with $\rho$ we obtain that

$$\partial_{t}M(t)=\langle \sigma_{kj}(t) \  | \ (\partial_{j}X^{k})\rho(t)\rangle-\langle p_{j}(t) \  | \ (\partial_{j}X^{k})p_k(t)\rangle+$$
$$\langle (-\Delta \rho(t)+2\rho^2(t)) \  | \ (\partial_{j}X^{j})\rho(t)\rangle$$
Now recalling that
$$\sigma_{jk}=\frac{1}{\rho}(p_jp_k+\partial_j \rho \partial_k \rho)$$
we have that
$$\partial_{t}M(t)=P_{1}+P_{2}+P_{3}+P_{4}$$
where
\begin{equation}\label{p1}
P_1:=\langle \rho^{-1}\partial_{k}\rho \partial_{j}\rho \  | \ (\partial_{j}X^{k})\rho(t)\rangle
\end{equation}
\\
\begin{equation}\label{p2}
P_2:=\langle \rho^{-1}p_kp_j \  | \ (\partial_{j}X^{k})\rho(t)\rangle-\langle p_j \  | \ (\partial_{j}X^{k})p_k\rangle
\end{equation}
\\
\begin{equation}\label{p3}
P_3:=\langle (-\Delta \rho) \  | \ (\partial_{j}X^{j})\rho(t)\rangle=\langle (-\Delta \rho) \  | \ (\nabla \cdot \vec{X})\rho(t)\rangle
\end{equation}
\\
\begin{equation}\label{p4}
P_4:=2\langle (\rho^2 \  | \ (\partial_{j}X^{j})\rho(t)\rangle=2\langle (\rho^2 \  | \ (\nabla \cdot \vec{X})\rho(t)\rangle
\end{equation}
\\
The term $P_1$ is clearly positive since $\partial_{j}X^{k}$ is a positive definite operator. Let's analyze $P_3$. Recalling that $-\Delta=D^2$
we have that
\\
$$P_3=\langle (-\Delta \rho) \  | \ (\nabla \cdot \vec{X})\rho(t)\rangle=(n-1)\langle (D^2\rho) \  | \ D^{-(n-1)}\rho(t)\rangle=$$
\\
$$(n-1)\langle D^{-\frac{n-3}{2}}\rho \  | \ D^{-\frac{n-3}{2}}\rho(t)\rangle=\frac{n-1}{2}\|D^{-\frac{n-3}{2}}(|u|^2)\|_{L^2}^2.$$ 
$P_4$ is also positive since
$$P_4=2\langle (\rho^2 \  | \ (\nabla \cdot \vec{X})\rho(t)\rangle=2(n-1)\langle (\rho^2 \  | \ D^{-(n-1)}\rho(t)\rangle \geq$$
\\
$$2(n-1)\langle (D^{-\frac{n-1}{2}}\rho^{\frac{3}{2}} \  | \ D^{-\frac{n-1}{2}}\rho^{\frac{3}{2}}\rangle \geq 0.$$
\\
Another way to inspect the positivity of this term is by explicitly expressing it as
\\
$$P_{4}=2\int_{\Bbb R^n \times \Bbb R^n}\frac{\rho^2(x)\rho(y)}{|x-y|}dxdy=\frac{1}{4}\int_{\Bbb R^n \times \Bbb R^n}\frac{|u(x)|^4|u(y)|^2}{|x-y|}dxdy
\geq 0.$$
\\
The only term that its positivity is not immediate is term $P_{2}$. Recall that
\\
$$(\partial_{j}X^{k})f(x)=\int_{\Bbb R^n}\eta_{kj}(x,y)f(y)dy$$
where the kernel $\eta_{kj}(x,y)$ is symmetric. Then
\\
$$P_2=\int_{\Bbb R^n \times \Bbb R^n}\{ \frac{\rho(y)}{\rho(x)}p_k(x)p_j(y)-p_k(y)p_j(x) \}\eta_{kj}(x,y)dxdy$$
\\
By changing variables we get
\\
$$P_2=\int_{\Bbb R^n \times \Bbb R^n}\{ \frac{\rho(x)}{\rho(y)}p_k(y)p_j(x)-p_k(x)p_j(y) \}\eta_{kj}(x,y)dxdy$$
\\
and thus 
\\
$$P_2=\frac{1}{2}\int_{\Bbb R^n \times \Bbb R^n}\{ \frac{\rho(y)}{\rho(x)}p_k(x)p_j(y)+\frac{\rho(x)}{\rho(y)}p_k(y)p_j(y)-p_j(x)p_k(y)-p_j(y)p_k(x) 
\}\eta_{kj}(x,y)dxdy=$$
\\
$$\frac{1}{2}\int_{\Bbb R^n \times \Bbb R^n}\{ \sqrt{\frac{\rho(y)}{\rho(x)}}p_k(x)-\sqrt{\frac{\rho(x)}{\rho(y)}}p_j(y)\}
\{\sqrt{\frac{\rho(y)}{\rho(x)}}p_j(x)-\sqrt{\frac{\rho(x)}{\rho(y)}}p_k(y)\}\eta_{kj}(x,y)dxdy.$$
\\
Thus if we define the two point momentum vector
$$\vec{J}(x,y)=\sqrt{\frac{\rho(y)}{\rho(x)}}\vec{p}(x)-\sqrt{\frac{\rho(x)}{\rho(y)}}\vec{p}(y)$$
we can write
$$P_2=\frac{1}{2}\langle J^{j}J_{k} \  | \ (\partial_{j}X^{k})\rangle \geq 0$$
since $\partial_{j}X^{k}$ is positive definite. 
\\
\\
We keep only $P_3$ and after integrating in time we have the main estimate of this paper which reads
$$\|D^{-\frac{n-3}{2}}(|u|^2)\|_{L_{t}^{2}L_{x}^2}^2 \lesssim \sup_{t} M(t).$$
\\
It remains to show that $M(t)$ is bounded by the appropriate norms. But
\\
$$M(t)=\langle [x;D^{-(n-1)}]_{j}\rho(t) \  | \ p_{j}(t)\rangle \lesssim \|p_{j}\|_{L^1}\|[x;D^{-(n-1)}]_j\rho(t)\|_{L^{\infty}} \lesssim$$
\\
$$\|p_{j}\|_{L^1}\|\rho\|_{L^1}\|[x;D^{-(n-1)}]_j\|_{L^1\rightarrow L^{\infty}}.$$
\\
Now by Hardy's inequality we have
$$\|p_{j}\|_{L^1}\lesssim \|u\|_{\dot{H}^{\frac{1}{2}}}^2$$
while
$$\|\rho\|_{L^1}=\frac{1}{2}\|u\|_{L^{2}}^2.$$
\\
Finally the operator norm $\|[x;D^{-(n-1)}]_j\|_{L^1\rightarrow L^{\infty}}$ is bounded by $1$ since for $f \in L^1$
\\
$$\vec{X}f(x)=\int_{\Bbb R^n}\frac{x-y}{|x-y|}f(y)dy.$$
Thus all in all we have that
\\
$$\|D^{-\frac{n-3}{2}}(|u|^2)\|_{L_{t}^{2}L_{x}^2}^2 \lesssim \|u\|_{L_{t}^{\infty}\dot{H}^{\frac{1}{2}}}^2\|u\|_{L_{t}^{\infty}L_{x}^2}^2$$
valid for all $n \geq 2$. In particular for $n=2$ the estimate reads
\\
$$\|D^{\frac{1}{2}}(|u|^2)\|_{L_{t}^{2}L_{x}^2}^2 \lesssim \|u\|_{L_{t}^{\infty}\dot{H}^{\frac{1}{2}}}^2\|u\|_{L_{t}^{\infty}L_{x}^2}^2$$
\\
which corresponds to the nonlinear diagonal case analogue of Bourgain's bilinear refinement of Strichartz estimate, \cite{jb1}. In this paper
we will use the following estimate in 2d
\\
\begin{equation}\label{basic}
\|u\|_{L_{t}^{4}L_{x}^{8}}^4\lesssim \|u\|_{L_{t}^{\infty}\dot{H}^{\frac{1}{2}}}^2\|u\|_{L_{t}^{\infty}L_{x}^2}^2
\end{equation} 
\\
which can be obtained by the previous estimate and the Sobolev embedding in two dimensions since
\\
$$\|u\|_{L_{t}^{4}L_{x}^{8}}^4=\||u|^{2}\|_{L_{t}^{2}L_{x}^{4}}^2 \lesssim \|D^{\frac{1}{2}}(|u|^{2})\|_{L_{t}^{2}L_{x}^{2}}^2 
\lesssim \|u\|_{L_{t}^{\infty}\dot{H}^{\frac{1}{2}}}^2\|u\|_{L_{t}^{\infty}L_{x}^2}^2.$$
\\
Note that the method we used is quite general. Thus we can consider operators of the form
$$\vec{X}:=[x:H]$$
where $H$ is a selfadjoint operator. The two crucial properties that we need is that $\partial_{j}X^{k}$ is positive and that we can bound the action
 $M(t)$ for a weight function $a(x)$. We will exploit these in a subsequent paper.
\subsection{Correlation estimates in one dimension. Proof of Theorem 2}
In this subsection we would like to prove the analogue of \eqref{nmor} in one dimension. Thus we show that
\begin{equation}\label{1d}
\|\partial_x(|u|^2)\|_{L_t^{\infty}L_x^2}^2 \lesssim \|u\|_{L_t^{\infty}L_x^2}^3\|u\|_{L_t^{\infty}\dot{H}_x^{1}}
\end{equation}
for solutions of the one dimensional NLS $iu_{t}+u_{xx}=|u|^{p-1}u$ for any $p$. Since this is a linear estimate as the proof will show
 the estimate is true for any power nonlinearity. We will do the calculations for $p=3$ but the same calculations
 establish \eqref{1d} for any power nonlinearity. We will follow the Gauss-Weierstrass summability method. The local conservation laws
 in one dimension can be written in the following form 

Mass conservation
\begin{equation}\label{lmass}
\partial_t \rho +\partial_x p=0
\end{equation} 
and momentum conservation
\\
\begin{equation}\label{lmoment}
\partial_t p+\partial_x \{2\rho^2-\rho_{xx}+\frac{1}{\rho}(p^2+\rho_{x}^2)\}=0
\end{equation} 
where $\rho=\frac{1}{2}|u|^2$ and $p=\Im (\bar{u}u_x)$.
\\
\\
Define the action 
$$M(t)=\int \int_{\Bbb R \times \Bbb R}a(x-y)\rho(y)p(x)dxdy$$
where 
$$a(x-y)=\erf(\frac{x-y}{\epsilon})=\int_{-\infty}^{\frac{x-y}{\epsilon}}e^{-t^2}dt$$ 
\\
is the scaled error function. This function is odd and bounded by $1$. Its derivative is 
$$\partial_x \erf(\frac{x-y}{\epsilon})=\frac{1}{\epsilon}e^{-\frac{(x-y)^2}{\epsilon^2}} \geq 0$$
which is the heat kernel in one dimension. It is immediate that
\\
$$\sup_{t}|M(t)| \lesssim \|u\|_{L_t^{\infty}L_x^2}^3\|u\|_{L_t^{\infty}\dot{H}_x^{1}}.$$
\\
Notice that the action $M(t)$ can be written as
\\
$$M(t)=\langle X\rho \ |\ p \rangle$$
\\
where the antisymmetric operator
\\
$$Xf(x)=(\erf(\frac{\cdot}{\epsilon})\star f)(x)=\int_{\Bbb R}\erf(\frac{x-y}{\epsilon})f(y)dy$$
\\
The derivative of this operator is the solution of the heat equation in one dimension
\\
$$X^{'}f(x)=\frac{1}{\epsilon}\int_{\Bbb R}e^{-\frac{(x-y)^2}{\epsilon^2}}f(y)dy$$
\\
with initial data the function $f(x)$. Since $X$ is antisymmetric and thus $\langle Xf\ |\ g\rangle=-\langle f\ |\ Xg\rangle$ by differentiating
 the action with respect to time we obtain
\\
$$\dot{M}(t)=\langle X\partial_{t}\rho \ |\ p \rangle + \langle X\rho \ |\ \partial_{t} p\rangle=
-\langle \partial_{t}\rho \ |\ Xp \rangle+\langle X\rho \ |\ \partial_{t}p \rangle.$$
\\
If we use the conservation laws \eqref{lmass} and \eqref{lmoment} and integrate by parts we have that
$$\dot{M}(t)=P_{1}+P_2+P_3+P_4$$
where
$$P_1=\langle X^{'}\rho \ |\ \frac{1}{\rho}\rho_{x}^2\rangle,\ \ \ P_4=\langle X^{'}\rho \ |\ 2\rho^2\rangle$$

$$P_3=\langle X^{'}\rho \ |\ -\rho_{xx}\rangle,\ \ \ P_2=\langle X^{'}\rho \ |\ \frac{1}{\rho}p^2\rangle-\langle X^{'}p \ |\ p\rangle$$
But
$$P_1=\int_{}\int{}\frac{1}{\epsilon}e^{-\frac{(x-y)^2}{\epsilon^2}}\frac{\rho(y)}{\rho(x)}\rho_{x}^2(x)dxdy \geq 0.$$

$$P_4=\int_{}\int{}\frac{2}{\epsilon}e^{-\frac{(x-y)^2}{\epsilon^2}}\rho(y)\rho(x)^2dxdy \geq 0.$$
\\
$$P_2=\int_{}\int{}\frac{1}{\epsilon}e^{-\frac{(x-y)^2}{\epsilon^2}}\left(\frac{\rho(y)}{\rho(x)}p^2(x)-p(x)p(y)\right)dxdy$$
and thus
$$2P_2=\int_{}\int{}\frac{1}{\epsilon}e^{-\frac{(x-y)^2}{\epsilon^2}}\left(\frac{\rho(y)}{\rho(x)}p^2(x)+\frac{\rho(x)}{\rho(y)}p^2(y)
-2p(x)p(y)\right)dxdy=$$
\\
$$\int_{}\int{}\frac{1}{\epsilon}e^{-\frac{(x-y)^2}{\epsilon^2}}\left(\sqrt{\frac{\rho(y)}{\rho(x)}}p(x)-\sqrt{\frac{\rho(x)}{\rho(y)}}p(y)\right)^2dxdy
\geq 0.$$
 Thus we have that
$$P_{3} \leq \dot{M}(t).$$
But
$$P_3=\int_{}\int{}\frac{1}{\epsilon}e^{-\frac{(x-y)^2}{\epsilon^2}}\rho(y)(-\rho_{xx}(x))dxdy \geq 0=$$
\\
$$=\int_{}(\frac{1}{\epsilon}e^{-(\frac{\cdot}{\epsilon})^2}\star \rho)(x)(-\rho_{xx}(x))dx=\int_{}\xi^2\hat{\rho}^2(\xi)e^{-\epsilon \xi^2}d\xi$$
by Plancherel's theorem. Sending $\epsilon \downarrow 0$ and integrating in time we obtain \eqref{1d}.
\\
\\
Actually more is true. Notice that since 
$$\lim_{\epsilon \rightarrow 0}(\frac{1}{\epsilon}e^{-(\frac{\cdot}{\epsilon})^2}\star \rho)(x)=\rho(x)$$
 we have
that
$$\lim_{\epsilon \rightarrow 0}P_{1}=\int_{}\rho_{x}^2(x)dx=\frac{1}{4}\|\partial_{x}(|u|^2)\|_{L_x^2}^2,$$
\\
$$\lim_{\epsilon \rightarrow 0}P_2=0$$
\\
$$\lim_{\epsilon \rightarrow 0}P_{4}=\frac{1}{4}\|u\|_{L_x^6}^6$$
\\
Notice that $P_{1}$ and $P_3$ are linear estimates and $P_4$ is the nonlinear estimate. Thus if we consider a nonlinearity of the form $|u|^{p-1}u$ we
 have that
$$\lim_{\epsilon \rightarrow 0}P_{4}=\frac{1}{2^{\frac{p+1}{2}}}\|u\|_{L_x^{p+3}}^{p+3}.$$
\\
This implies that for the solutions of $iu_{t}+u_{xx}=|u|^{p-1}u$ we obtain the following a priori 1d estimate
\\
$$\|u\|_{L_t^{p+3}L_x^{p+3}}^{p+3}\lesssim \|u\|_{L_t^{\infty}L_x^2}^3\|u\|_{L_t^{\infty}\dot{H}_x^{1}}.$$
\\
Recalling that the scaling is $$u^{\lambda}(x,t)=\lambda^{-\frac{2}{p-1}}u(\frac{x}{\lambda},\frac{t}{\lambda})$$
we can easily verify that the above estimate is scale invariant.

\section{$H^1$ scattering. Proof of Theorem 3}
In this section we prove Theorem 3. As we have said the first proof of
this result was obtained in \cite{kn} with a more complicated argument
using induction on energy. An analogous simplified proof of scattering
for the $L^2$-supercritical NLS problems in one space dimension
appears in \cite{chvz}.  
What we have shown so far is that for solutions of the \eqref{nls} in two dimensions the following global a priori estimate is true
\\
\begin{equation}\label{corr}
\|D^{\frac{1}{2}}(|u|^{2})\|_{L_t^2L_x^2} \lesssim \|u\|_{L_t^{\infty}\dot{H}_{x}^{\frac{1}{2}}}\|u\|_{L_t^{\infty}L_x^2}.
\end{equation}
%\\
%We can use this estimate to obtain a global a priori bound in mixed Lebesgue spaces as follows
%\\
%$$\|u\|_{L_x^8}^{4}=\||u|^2\|_{L_x^4}^{2} \lesssim \|D^{\frac{1}{2}}(|u|^{2})\|_{L_x^2}^{2}$$\\
%where in the last inequality we
% applied Sobolev embedding. Now if we integrate in time we have
%\\
%$$\|u\|_{L_t^4L_x^8}^{4}=\|\|u\|_{L_x^8}^{4}\|_{L_t^{1}} \lesssim \int_{0}^{T}\|D^{\frac{1}{2}}(|u|^{2})\|_{L_x^2}^{2}dt$$
%\\
As we have already mentioned by Sobolev embeding and using \eqref{corr} we obtain that
\begin{equation}\label{2dlaw}
\|u\|_{L_t^4L_x^8}^{4} \lesssim \|u\|_{L_t^{\infty}\dot{H}_{x}^{\frac{1}{2}}}^{2}\|u\|_{L_t^{\infty}L_x^2}^{2}.
\end{equation}
\\
By conservation of energy and mass the estimate implies that
\begin{equation}\label{first}
\|u\|_{L_t^4L_x^8} \lesssim C_{E(u_{0})}
\end{equation}

To prove scattering we have to upgrade this control to Strichartz control. Define the norms
$$\|u\|_{S^{1}}:=\sup_{\frac{1}{q}+\frac{1}{r}=\frac{1}{2}} \|\langle \nabla \rangle u\|_{S^{0}}.$$ 
Assume that we have
\\
$$\|u\|_{L_t^4L_x^8} \lesssim C_{E(u_{0})}.$$
\\
Divide the real line into finitely many sub-intervals $I_{j}$ such that on each $I_{j}$ we have that
\\
$$\|u\|_{L_{t\in I_{j}}^4L_x^8} \sim \delta.$$
\\ 
We will show that on each $I_{j}$ we have the bound
\\
\begin{equation}\label{lwp}
\|u\|_{S^{1}(I_{j})} \lesssim \|u_{0}\|_{H^{1}}.
\end{equation}
\\
Since there are only finitely many $I_{j}$'s we have 
$$\|u\|_{S^{1}}\lesssim C_{E}$$
\\
and thus scattering follows by standard arguments. Thus it remains to prove \eqref{lwp}.
\\
\\
We will suppress the $I_{j}$ notation for what follows.
By Duhamel's formula we have 
$$u(x,t)=e^{it\Delta}u_{0}-i\int_{0}^{T}e^{i(t-s)\Delta}(|u|^{p-1}u)(s)ds.$$
By Lemma \ref{linstr} and H\"older's inequality we have that
\\ 
$$\|u\|_{S^{1}}\lesssim \|u_{0}\|_{H^{1}}+\|\langle \nabla \rangle (|u|^{p-1}u)\|_{L{_t}^{\frac{4}{3}}L_{x}^{\frac{4}{3}}} \lesssim 
\|u_{0}\|_{H^{1}}+ \||u|^{p-1}(\langle \nabla \rangle u)\|_{L{_t}^{\frac{4}{3}}L_{x}^{\frac{4}{3}}} $$
\\
$$\lesssim \|u_{0}\|_{H^{1}}+ \|\langle \nabla \rangle u\|_{L_{t}^{\infty}L_{x}^{2}}\|u^{p-1}\|_{L_{t}^{\frac{4}{3}}L_{x}^{4}}
\lesssim \|u_{0}\|_{H^{1}}+ \|u\|_{S^{1}}\|u^{p-1}\|_{L_{t}^{\frac{4}{3}}L_{x}^{4}}$$
\\
$$\lesssim \|u_{0}\|_{H^{1}} + 
\|u\|_{S^{1}}\|u\|_{L_t^4L_x^8}^{\epsilon}\|u\|_{L_{t}^{\frac{4(p-1-\epsilon)}{3-\epsilon}}L_{x}^{\frac{8(p-1-\epsilon)}
{2-\epsilon}}}^{p-1-\epsilon}.$$
\\
This last inequality follows by the interpolation of the $L_{p}$ spaces. Thus
\\ 
$$\|u\|_{S^{1}}\lesssim \|u_{0}\|_{H^{1}}+\delta^{\epsilon}\|u\|_{S^{1}}\|u\|_{L_{t}^{\frac{4(p-1-\epsilon)}{3-\epsilon}}L_{x}^{\frac{8(p-1-\epsilon)}
{2-\epsilon}}}^{p-1-\epsilon}.$$
\\
Now we apply Sobolev embedding
\\
$$\|u\|_{L_{t}^{\frac{4(p-1-\epsilon)}{3-\epsilon}}L_{x}^{\frac{8(p-1-\epsilon)}
{2-\epsilon}}} \lesssim \||\nabla|^{\alpha}u\|_{L_{t}^{\frac{4(p-1-\epsilon)}{3-\epsilon}}L_{x}^{\frac{4(p-1-\epsilon)}
{2p-5-\epsilon}}}$$
\\
where 
$$\alpha=\frac{p-3-\frac{\epsilon}{4}}{p-1-\epsilon}.$$
\\
Note to apply the Sobolev embedding we must have $\frac{8(p-1-\epsilon)}{2-\epsilon}>\frac{4(p-1-\epsilon)}
{2p-5-\epsilon}$ a restriction which gives $p>3+\frac{\epsilon}{4}$ which is acceptable. For the same reason $\alpha>0$. Finally
note that the pair $(\frac{4(p-1-\epsilon)}{3-\epsilon},\frac{4(p-1-\epsilon)}
{2p-5-\epsilon})$ is Strichartz admissible and thus since $\alpha<1$ we have that
$$\||\nabla|^{\alpha}u\|_{L_{t}^{\frac{4(p-1-\epsilon)}{3-\epsilon}}L_{x}^{\frac{4(p-1-\epsilon)}
{2p-5-\epsilon}}}\lesssim \|u\|_{S^{1}}.$$
All in all we have
$$\|u\|_{S^{1}}\lesssim \|u_{0}\|_{H^{1}} +\delta^{\epsilon}\|u\|_{S^{1}}^{p-\epsilon}$$
\\
and by a continuity argument for $\epsilon$ small we obtain 
\begin{equation}\label{second}
\|u\|_{S^{1}}\lesssim C_{E}.
\end{equation}
\\
We now use this estimate to prove asymptotic completeness, that is, there exist unique $u_{\pm}$ such that
\begin{equation}\label{scat lim}
\|u(t)-e^{it\Delta}u_{\pm}\|_{H^1(\Bbb R^2)}\to 0 \quad \text{as }t\to\pm\infty.
\end{equation}
\\
By time reversal symmetry, it suffices to prove the claim for positive times only.  For $t>0$, we define $v(t):=e^{-it\Delta}u(t)$.
We will  show that $v(t)$ converges in $H^1_x$ as $t\to +\infty$, and define $u_+$ to be the limit. Indeed, by Duhamel's formula,
\\
\begin{equation}\label{def v}
v(t)=u_0-i\int_0^t e^{-is\Delta} \bigl(|u|^{p-1} u\bigr)(s)\,ds.
\end{equation}
Therefore, for $0<\tau<t$,
$$
v(t)-v(\tau)=-i\int_{\tau}^t e^{-is\Delta}\bigl(|u|^{p-1}u\bigr)(s)\,ds.
$$
Arguing as above, by Lemma \ref{linstr} and Sobolev embedding,
\\
$$\|v(t)-v(\tau)\|_{H^1(\Bbb R^2)} \lesssim \|\langle \nabla \rangle \bigl(|u|^{p-1}u\bigr)\|_{L_t^{\frac{4}{3}}L_x^{\frac{4}{3}}([t,\tau]\times\Bbb R^2)}$$
$$\lesssim \|u\|_{L_{t \in [t,\tau]^4}L_x^8}^\epsilon \|\langle \nabla \rangle u\|_{S^0([t,\tau])}^{p-\epsilon}.$$
\\
Thus, by \eqref{first} and \eqref{second},
\\
$$\|v(t)-v(\tau)\|_{H^1(\R)}\to 0 \mbox{  as } \tau,t\to \infty.$$
\\
In particular, this implies $u_+$ is well defined and inspecting \eqref{def v} we find
$$u_+=u_0-i\int_0^{\infty}e^{-is\Delta}(|u|^{p-1}u)(s)\, ds.$$
\\
Using the same estimates as above, it is now an easy matter to derive \eqref{scat lim}.  This completes the proof of Theorem 3.

\section{Proof of  Theorem 5 and comments on further refinements.}
There is a problem when one tries to employ the strategy of Section 4 to prove Theorem 4. To prove that the problem is globally well-posed
 and that it scatters we have to obtain a priori control on the Strichartz norms. The idea is to upgrade \eqref{2dlaw} 
to obtain control on all the relevant Strichartz norms. The problem is that for solutions below the energy space the right
 hand side of \eqref{2dlaw} is not bounded anymore. Recall that to prove Theorem 3 we used strongly the fact that the $H^1$ norm of the solutions
 was bounded. Then we used this bound along with the estimate \eqref{2dlaw} to bound the $S^1$ norm of the solutions. Thus to prove
 Theorem 4 we have to bound the $H^s$ norm of the solution uniformly in time for $s<1$ and then use this bound along with \eqref{2dlaw}. 
The $H^1$ bound came from conservation of energy and we do not have at the moment a conserved quantity at the $H^s$ level. But we can define
 a new functional 
\begin{equation}\label{modenergy}
E(Iu)(t)=\frac{1}{2}\int |\nabla Iu(t)|^{2}dx+\frac{1}{p+1}\int |Iu(t)|^{p+1}dx=E(Iu_{0}). 
\end{equation}
where $Iu$ is a solution to the initial value problem
\begin{equation}\label{Inls}
\left\{
\begin{matrix}
iIu_{t}+ \Delta Iu -I(|u|^{2k}u)=0, & x \in {\mathbb R^2}, & t\in {\mathbb R},\\
Iu(x,0)=Iu_{0}(x)\in H^{s}({\mathbb R^2}).
\end{matrix}
\right.
\end{equation}
Note that $Iu$ solves the original equation \eqref{nls} up to an error 
$$I(|u|^{2k}u)-|Iu|^{2k}Iu.$$
Because of this we expect the functional $E(Iu)$ to be "almost conserved'' in the sense that its derivative will decay with respect to a large parameter. 
This will allow us to control $E(Iu)$ in time intervals that the local solutions are well-posed and we can iterate this control to
 obtain control globally in time. Then immediately we obtain a bound for the $H^1$ norm of $Iu$ which by Lemma \ref{basic property} will give us
an $H^s$ bound for the solutions $u$. In this process we will strongly use \eqref{2dlaw}. On the other hand, to be able to use
 \eqref{2dlaw} we need to have $H^s$ control on the norm of $u$. This feedback argument can be successfully implemented with the
 help of a standard continuity argument and this will be the contenx of this section. We will follow closely the argument in \cite{chvz}. 

We start by showing that the functional $E(Iu)$ is almost conserved. We need to define new norms. We fix $t \in [t_{0},T]$ and define
$$\|u\|_{Z(t)}: =\sup_{(q,r) \ \text{admissible}} \Bigl(\sum_{N\geq 1}\|\nabla P_N u\|^2_{L_t^qL_x^r([t_0,t]\times\R)}\Bigr)^{1/2}$$
\\
with the convention that $P_{1}u=P_{\leq 1}u$. We observe the inequality
\\
\begin{equation}\label{sqsum}
\Bigl\|\Bigl(\sum_{N\in 2^\Bbb Z} |f_N|^2 \Bigr)^{1/2}\Bigr\|_{L_t^qL_x^r} \leq \Bigl(\sum_{N\in 2^\Bbb Z} \|f_N\|^2_{L_t^qL_x^r}\Bigr)^{1/2}
\end{equation}
\\
for all $2 \leq q,r \leq \infty$ and arbitrary functions $f_N$, which one proves by interpolating between the trivial
cases $(2,2)$, $(2,\infty)$, $(\infty,2)$, and $(\infty,\infty)$. In particular, \eqref{sqsum} holds for all
admissible exponents $(q,r)$.  Combining this with the Littlewood-Paley inequality, we find
\\
\begin{align*}
\| u \|_{L_t^qL_x^r}
\lesssim \Bigl\|\Bigl(\sum_{N\in 2^\Bbb Z} |P_N u|^2\Bigr)^{1/2}\Bigr\|_{L_t^qL_x^r}
\lesssim \Bigl(\sum_{N\in 2^\Bbb Z} \|P_N u \|^2_{L_t^qL_x^r}\Bigr)^{1/2}.
\end{align*}
\\
In particular,
$$
\|\nabla u\|_{S^0([t_0,t])}\lesssim \|u\|_{Z(t)}.
$$
\\
The appearance of the homogeneous derivative in our definition of the space $Z(t)$ instead of the non-homogeneous derivative operator 
$\langle \nabla \rangle$ that we used in \cite{cgt} is imposed by the level of the criticality. That means that as the problem is
 $L^2$-supercritical, the $L^2$ norm of $Iu^{\lambda}$ grows as $\lambda$ grows. Thus using scaling we cannot control the full $H^1$ norm 
of the rescaled solution. This is the reason that we define the $Z$ norm as the homogeneous part of the $H^1$ norm. The reader can notice
 that we control all subsequent quantities by the homogeneous part of the $H^1$ norm where scaling works in our favor. 

The dual estimate of \eqref{sqsum} is 
\\
\begin{equation}\label{sqsumdual}
\Bigl(\sum_{N\in 2^\Bbb Z} \|f_N\|^2_{L_{t}^{q^{\prime}}L_{x}^{r^{\prime}}}\Bigr)^{1/2} 
\leq \Bigl\|\Bigl(\sum_{N\in 2^\Bbb Z} |f_N|^2 \Bigr)^{1/2}\Bigr\|_{L_{t}^{q^{\prime}}L_x^r{^{\prime}}} 
\end{equation} 
\\
Since the Littlewood-Paley operators commute with $i\partial_+\Delta$ by Lemma \ref{linstr} we have that
\begin{equation}
\||\nabla|P_{N}u\|_{ S^0(I)} \lesssim \|u(t_0)\|_{\dot{H}_x^s} + \||\nabla|P_{N}F\|_{L_t^{q_i'}L_x^{r_i'}(I\times\Bbb R^n)}.
\end{equation}
\\
Thus
$$\|u\|_{Z(t)}: =\sup_{(q,r) \ \text{admissible}} \Bigl(\sum_{N\geq 1}\|\nabla P_N u\|^2_{L_t^qL_x^r([t_0,t]\times\Bbb R)}\Bigr)^{1/2} \lesssim$$
\\
$$ \|u(t_0)\|_{\dot{H}_x^s}+\Bigl(\sum_{N\in 2^\Bbb Z} \||\nabla|P_{N}(i\partial_{t}+\Delta)\|^2_{L_{t}^{q^{\prime}}L_{x}^{r^{\prime}}}\Bigr)^{1/2}
\lesssim $$
\\
$$\|u(t_0)\|_{\dot{H}_x^s}+\Bigl\|\Bigl(\sum_{N\in 2^\Bbb Z} |P_{N}|\nabla|(i\partial_{t}+\Delta)|^2 \Bigr)^{1/2}\Bigr\|_{L_{t}^{q^{\prime}}L_x^r{^{\prime}}}$$
\\
where in the last inequality we applied \eqref{sqsumdual}. Thus if we applied the Littlewood-Paley theorem in this last inequality we have that
\\
\begin{equation}\label{localZ}
\|u\|_{Z(t)}\lesssim \|u(t_0)\|_{\dot{H}_x^s}+\||\nabla|(i\partial_{t}+\Delta)\|_{L_{t}^{q^{\prime}}L_x^r{^{\prime}}}.
\end{equation}
Now we define $Z_{I}(t)=\|Iu\|_{Z(t)}$.
\begin{prop}\label{aclocallaw}
Let $s>1-\frac{1}{2k-1}$ , $k \geq 2, \ \ k \in \Bbb N$, 
and let $u$ be an $H_x^s$ solution to \eqref{nls} on the spacetime slab $[t_0,T]\times \Bbb R^2$ with $E(I u(t_0))\le 1$.
Suppose in addition that
\begin{align}\label{ms}
\|u\|_{L_{t \in [t_0,T]}^4L_{x}^{8}}\le \eta
\end{align}
for a sufficiently small $\eta>0$ (depending on $k$ and on $E(I u(t_0))$). Then we have
\\
\begin{equation}\label{zit control}
Z_I(t)\lesssim \|\nabla Iu(t_0)\|_2+N^{-2}Z_I(t)^{2k+1}+\eta^{2}Z_I(t)^{2k-1}+\eta^{2}\sup_{s\in [t_0,t]} E(I_Nu(s))^{\frac {k-1}{k+1}}Z_I(t). 
\end{equation}
\\
\end{prop}
\begin{proof}
Throughout this proof, all spacetime norms are on $[t_0,t]\times \Bbb R^2$.  By \eqref{localZ} and H\"older's inequality, combined with the fact
that $\nabla I$ acts as a derivative (as the multiplier of $\nabla I$ is increasing in $|\xi|$), we estimate
\\
\begin{equation}\label{l1}
Z_I(t)\lesssim \|\nabla Iu(t_0)\|_2+\|\nabla I(|u|^{2k}u)\|_{L_{t}^{\frac{4}{3}}L_{x}^{\frac{4}{3}}} \lesssim
\end{equation} 
$$\|\nabla Iu(t_0)\|_2+\|u\|_{4k,4k}^{2k}\|\nabla Iu\|_{4,4} \lesssim \|\nabla Iu(t_0)\|_2+\|u\|_{4k,4k}^{2k}Z_I(t).$$
\\
To estimate $\|u\|_{4k,4k}$, we decompose $u:=u_{\le 1}+u_{1<\cdot \le N}+u_{>N}$.  To estimate the low frequencies we use interpolation and obtain
\\
$$\|u_{\leq 1}\|_{4k,4k}^{2k} \lesssim \|u\|_{L_{t}^{4}L_{x}^{8}}^2\|u_{\leq 1}\|_{L_{t}^{\infty}L_{x}^{8(k-1)}}^{2(k-1)}$$
\\
Since for $k \geq 2$ we have that $8(k-1)>2k+2$ by Bernstein's inequality we have that
$$\|u_{\leq 1}\|_{L_{t}^{\infty}L_{x}^{8(k-1)}}\lesssim \|u_{\leq 1}\|_{L_{t}^{\infty}L_{x}^{2k+2}} \lesssim E(Iu)^{\frac{1}{2k+2}}$$
where we use the energy bound. Thus
\begin{equation}\label{l2}
\|u_{\leq 1}\|_{4k,4k}^{2k}\lesssim \eta^2\sup_{s\in[t_{0},t]}E(Iu)^{\frac{k-1}{k+1}}.
\end{equation}
For the medium frequencies again by interpolation we have
$$\|u_{1<\cdot \le N}\|_{4k,4k}^{2k} \lesssim \|u\|_{L_{t}^{4}L_{x}^{8}}^2\|u_{1<\cdot \le N}\|_{L_{t}^{\infty}L_{x}^{8(k-1)}}^{2(k-1)}$$
But by Sobolev embedding
$$\|u_{1<\cdot \le N}\|_{L_{t}^{\infty}L_{x}^{8(k-1)}}^{2(k-1)}\lesssim \||\nabla|^{\frac{4k-5}{4k-4}}u_{1<\cdot \le N}\|_{L_{t}^{\infty}L_x^2}^{2(k-1)}
\lesssim \|\nabla Iu\|_{L_{t}^{\infty}L_x^2}^{2(k-1)}$$
and thus
\begin{equation}\label{l3}
\|u_{1<\cdot \le N}\|_{4k,4k}^{2k} \lesssim \eta^2 Z_{I}(t)^{2(k-1)}.
\end{equation}
Finally to estimate the high frequencies we apply Sobolev embedding and Lemma \ref{basic property} to obtain
\\
$$\|u_{>N}\|_{4k,4k}^{2k} \lesssim \||\nabla|^{1-\frac{1}{k}}u_{>N}\|_{L_{t}^{4k}L_{x}^{\frac{4k}{2k-1}}}^{2k} 
\lesssim N^{-2}\|\nabla Iu\|_{L_{t}^{4k}L_{x}^{\frac{4k}{2k-1}}}^{2k}$$
Since the pair $(4k,\frac{4k}{2k-1})$ is admissible we obtain
\begin{equation}\label{l4}
\|u_{>N}\|_{4k,4k}^{2k}\lesssim N^{-2}Z_{I}(t)^{2k}.
\end{equation}
\\
Using \eqref{l1},\eqref{l2}, \eqref{l3}, and \eqref{l4} we obtain the Proposition.
\end{proof}
\begin{prop}\label{aclaw}
Let $s>1-\frac{1}{2k-1}$ , $k \geq 2, \ \ k \in \Bbb N$, 
and let $u$ be an $H_x^s$ solution to \eqref{nls} on the spacetime slab $[t_0,T]\times \Bbb R^2$ with $E(I u(t_0))\le 1$.
Suppose in addition that
\begin{align}\label{eta}
\|u\|_{L_{t \in [t_0,T]}^4L_{x}^{8}}\le \eta
\end{align}
for a sufficiently small $\eta>0$ (depending on $k$ and on $E(I u(t_0))$). Then we have
\\
\begin{align}
\bigl| & \sup_{s\in [t_0,t]} E(Iu(s))-E(Iu(t_0))\bigr|\label{energy increment} \\
&\lesssim N^{-1+}\Bigl(Z_I(t)^{2k+2}+ \eta^{2}Z_I(t)^2\sup_{s\in [t_0,t]} E(I u(s))^{\frac {k-1}{k+1}}\notag\\
&\qquad \qquad  + \sum_{J=3}^{2k+2}\eta^{\frac {2k+2-J}{2k-1}}Z_I(t)^J \sup_{s\in [t_0,t]} E(Iu(s))^{\frac {(k-1)(2k+2-J)}{(2k-1)(k+1)}}\Bigr)\notag\\
&\quad +N^{-1+}\Bigl(Z_{I}(t)^{2k+1}+ \eta^{2}Z_I(t)\sup_{s\in [t_0,t]} E(Iu(s))^{\frac {k-1}{k+1}}\Bigr)\notag\\
&\qquad \qquad\times \Bigl(Z_{I}(t)^{2k+1}+\eta\sup_{s\in [t_0,t]} E(Iu(s))^{\frac {k}{k+1}}\Bigr)\notag\\
&\quad +N^{-1+}\sum_{J=3}^{2k+2}\eta^{\frac{2k+2-J}{2k-1}}Z_{I}(t)^{J-1}\sup_{s\in [t_0,t]} E(Iu(s))^{\frac {(k-1)(2k+2-J)}{(2k-1)(k+1)}}\notag\\
&\qquad\qquad \times \Bigl(Z_{I}(t)^{2k+1}+\eta \sup_{s\in [t_0,t]} E(Iu(s))^{\frac {k}{k+1}}\Bigr).\notag
\end{align}
\end{prop}
\begin{proof}
As
\begin{align*}
\frac d{dt}E(u(t))=\Re \int \bar u_t(|u|^{2k}u-\Delta u)\,dx=\Re \int \bar u_t(|u|^{2k}u-\Delta u-iu_t)\,dx,
\end{align*}
we obtain
\begin{align*}
\frac d{dt}E(Iu(t))&=\Re\int I\bar u_t (|Iu|^{2k}Iu-\Delta I u-iIu_t)\,dx\\
&=\Re \int I\bar u_t (|Iu|^{2k}Iu-I(|u|^{2k}u))\,dx.
\end{align*}
Using the Fundamental Theorem of Calculus and Plancherel, we write
\\
\begin{align*}
E(Iu(t))&-E(Iu(t_0))\\
&=\Re\itxi \Bigl(\symb\Bigr)\\
&\qquad\qquad\qquad \widehat {\overline{I\partial_tu}}(\xi_1) \widehat{Iu}(\xi_2)\cdots\widehat{\overline{Iu}}(\xi_{2k+1}) 
\widehat{Iu}(\xi_{2k+2})\,d\sigma(\xi)\,ds.
\end{align*}
As $iu_t=-\Delta u+|u|^{2k}u$, we thus need to control
\\
\begin{align}\label{term1}
\Bigl|\itxi&\Bigl(\symb\Bigr)\\
&\Delta \widehat{\overline {Iu}}(\xi_1)\widehat {Iu}(\xi_2)\cdots \widehat{\overline{Iu}}(\xi_{2k+1})\widehat {Iu}(\xi_{2k+2})\,d\sigma(\xi)\,ds\Bigr|\notag
\end{align}
and
\\
\begin{align}\label{term2}
\Bigl|\itxi&\Bigl(\symb\Bigr)\\
&\widehat {\overline{I(|u|^{2k}u)}} (\xi_1)\widehat{Iu}(\xi_2)\cdots \widehat{\overline{Iu}}(\xi_{2k+1})\widehat {Iu}(\xi_{2k+2})\,d\sigma(\xi)\,ds\Bigr|.\notag
\end{align}

We first estimate \textbf{\eqref{term1}}.  To this end, we decompose
$$
u:=\sum_{N\geq 1}P_N u
$$
with the convention that $P_1 u: =P_{\leq 1}u$.  Using this notation and symmetry, we estimate
\begin{equation}
\eqref{term1}\lesssim \sum_{\substack{N_1, \dots, N_{2k+2}\geq 1 \\N_2\ge N_3\ge\cdots\geq N_{2k+2}}}B(N_1,\dots,N_{2k+2}),\label{al2}
\end{equation}
where
\begin{align*}
&B(N_1,\dots,N_{2k+2})\\
&\qquad:=\Bigl|\itxi\Bigl(\symb\Bigr)    \\
&\qquad \qquad \qquad \Delta\widehat{\overline{Iu_{N_1}}}(\xi_1)\widehat{Iu_{N_2}}(\xi_2)\cdots 
\widehat{\overline{Iu_{N_{2k+1}}}}(\xi_{2k+2}) \widehat {Iu_{N_{2k+2}}}(\xi_{2k+2})d\sigma(\xi)ds\Bigr|.
\end{align*}
\\
\\
\noindent{\bf Case $I$:} $N_1> 1$, $N_2\ge \cdots \ge N_{2k+2}> 1$.
\\
\\
\noindent \textbf{Case $I_a$:}  $N\gg N_2$.
\\
In this case,
$$
m(\xi_2+\xi_3+\cdots+\xi_{2k+2})=m(\xi_2)=\dots=m(\xi_{2k+2})=1.
$$
Thus,
$$B(N_1,\dots,N_{2k+2})=0$$
and the contribution to the right-hand side of \eqref{al2} is zero.
\\
\\
\noindent \textbf{Case $I_b$:} $N_2\gtrsim N\gg N_3$.
\\
\\
As $\sum_{i=1}^{2k+2}\xi_i=0$, we must have $N_1\sim N_2$.  Thus, by the Fundamental Theorem of Calculus,
\\
\begin{align*}
\Bigl|\symb\Bigr|&=\Bigl|1-\frac {m(\xi_2+\cdots+\xi_{2k+2})}{m(\xi_2)}\Bigr| \\
&\lesssim \Bigl|\frac {\nabla m(\xi_2)(\xi_3+\cdots+\xi_{2k+2})}{m(\xi_2)}\Bigr|\lesssim \frac {N_3}{N_2}.
\end{align*}
\\
Applying the multilinear multiplier theorem of Coifman and Meyer
(cf. \cite{cmfourier}, and \cite{cmaster}), Sobolev embedding, Bernstein, and recalling that $N_j>1$, we estimate
\begin{align*}
&B(N_1,\dots,N_{2k+2}) \\
&\qquad \lesssim \frac {N_3}{N_2}\|\Delta Iu_{N_1}\|_{4,4}\|Iu_{N_2}\|_{4,4}\|Iu_{N_3}\|_{4,4}\prod_{j=4}^{2k+2}\|Iu_{N_j}\|_{4(2k-1),4(2k-1)}\\
&\qquad \lesssim \frac {N_1}{N_2^2}\prod_{j=1}^3\|\nabla I u_{N_j}\|_{4,4} \prod_{j=4}^{2k+2}\||\nabla|^{\frac {k-2}{2k-1}} Iu_{N_j}\|_{4(2k-1),\frac {4(2k-1)}{4k-3}}\\
&\qquad \lesssim \frac 1{N_2}Z_I(t)^{2k+2}
\lesssim N^{-1+}N_2^{0-}Z_I(t)^{2k+2}.
\end{align*}
The factor $N_2^{0-}$ allows us to sum in $N_1,N_2,\dots, N_{2k+2}$, this case contributing at most $N^{-1+}Z_I(t)^{2k+2}$ to the
right-hand side of \eqref{al2}.
\\
\\
\noindent \textbf{Case $I_c$:} $N_2\gg N_3\gtrsim N$.
\\
\\
As $\sum_{i=1}^{2k+2}\xi_i=0$, we must have $N_1\sim N_2$.  Thus, as $m$ is decreasing,
\\
$$\Bigl|\symb\Bigr|\lesssim \frac {m(\xi_1)}{m(\xi_2)\cdots m(\xi_{2k+2})}.$$
\\
Using again the multilinear multiplier theorem, Sobolev embedding, Bernstein, and the fact that $m(\xi)|\xi|^{\frac{1}{2k-1}}$ is increasing
for $s>1-\frac{1}{2k-1}$, we estimate
\\
\begin{align*}
& \ B(N_1,\dots,N_{2k+2})   \\
&\lesssim \frac {m(N_1)}{m(N_2)\cdots m(N_{2k+2})}\frac {N_1}{N_2N_3} \prod_{j=1}^3\|\nabla I u_{N_j}\|_{4,4}
\prod_{j=4}^{2k+2}\||\nabla|^{\frac {2(k-1)}{2k-1}}Iu_{N_j}\|_ {4(2k-1),\frac {4(2k-1)}{4k-3}}\\
&\lesssim \frac 1{N_3m(N_3)\prod_{j=4}^{2k+2}m(N_j)N_j^{\frac {1}{2k-1}}}\prod_{j=1}^3\|\nabla I u_{N_j}\|_{4,4} 
\prod_{j=4}^{2k+2}\|\nabla I u_{N_j}\|_{4(2k-1),\frac{4(2k-1)}{4k-3}}\\
&\lesssim \frac 1{N_3m(N_3)}\|\nabla I u_{N_1}\|_{4,4} \|\nabla I u_{N_2}\|_{4,4}Z_I(t)^{2k}\\
&\lesssim N^{-1+}N_3^{0-}\|\nabla I u_{N_1}\|_{4,4} \|\nabla I u_{N_2}\|_{4,4}Z_I(t)^{2k}.
\end{align*}
The factor $N_3^{0-}$ allows us to sum over $N_3, \dots ,N_{2k+2}$.  To sum over $N_1$ and $N_2$, we use the fact that $N_1\sim N_2$
and Cauchy-Schwarz to estimate the contribution to the right-hand side of \eqref{al2} by
\begin{align*}
N^{-1+}\Bigl(\sum_{N_1>1}\|\nabla Iu_{N_1}\|_{4,4}^2\Bigr)^{\frac 12}\Bigl(\sum_{N_2>1}\|\nabla I u_{N_2}\|_{4,4}^2\Bigr)^{\frac 12}Z_I(t)^{2k}
\lesssim N^{-1+}Z_I(t)^{2k+2}.
\end{align*}
\\
\noindent \textbf{Case $I_d$:} $N_2\sim N_3\gtrsim N$.
\\
\\
As $\sum_{i=1}^{2k+2}\xi_i=0$, we obtain $N_1\lesssim N_2$, and hence $m(N_1)\gtrsim m(N_2)$ and $m(N_1)N_1\lesssim m(N_2)N_2$.  Thus,
$$
\Bigl|\symb\Bigr|\lesssim \frac {m(N_1)}{m(N_2)m(N_3)\cdots m(N_{2k+2})}.
$$
Arguing as for Case $I_c$, we estimate
\begin{align*}
B(N_1,\dots, N_{2k+2})&\lesssim \frac {m(N_1)N_1}{m(N_2)N_2m(N_3)N_3\prod_{j=4}^{2k+2}m(N_j)N_j^{\frac
{1}{2k-1}}}Z_I(t)^{2k+2}\\
&\lesssim \frac 1{m(N_3)N_3}Z_I(t)^{2k+2}\\
&\lesssim N^{-1+}N_3^{0-}Z_I(t)^{2k+2}.
\end{align*}
The factor $N_3^{0-}$ allows us to sum over $N_1, \dots, N_{2k+2}$.  This case contributes at most $N^{-1+}Z_I(t)^{2k+2}$
to the right-hand side of \eqref{al2}.
\\
\\
\noindent \textbf{Case $II$:} There exists $1\leq j_0\leq 2k+2$ such that $N_{j_0}=1$.  Recall that by our convention, $P_1:=P_{\leq 1}$.
\\
\\
\noindent \textbf{Case $II_a$:} $N_1=1$.
\\
\\
Let $J$ be such that $N_2\ge \dots\geq N_J> 1=N_{J+1}=\dots=N_{2k+2}$.  Note that we may assume $J\geq 3$ since otherwise
$$
B(N_1,\dots,N_{2k+2})=0.
$$
Also, arguing as for Case $I_a$, if $N\gg N_2$ then
$$
B(N_1,\dots,N_{2k+2})=0.
$$
Thus, we may assume $N_2\gtrsim N$.  In this case we cannot have $N_2\gg N_3$ since it would contradict $\sum_{i=1}^{2k+2}\xi_i=0$ and $N_1=1$.
Hence, we must have
$$
N_2\sim N_3\gtrsim N.
$$
As
$$
\Bigl|\symb\Bigr|\lesssim \frac 1{m(N_2)m(N_3)\cdots m(N_{2k+2})},
$$
we use the multilinear multiplier theorem and Sobolev embedding to estimate
\begin{align*}
& \ B(N_1,\dots,N_{2k+2}) \\
&\lesssim \frac {N_1}{m(N_2)N_2m(N_3)N_3m(N_4)\cdots m(N_{2k+2})}\prod_{j=1}^3\|\nabla Iu_{N_j}\|_{4,4}\\
&\qquad \qquad \times \prod_{j=4}^J\||\nabla|^{\frac {2(k-1)}{2k-1}}Iu_{N_j}\|_{4(2k-1),\frac {4(2k-1)}{4k-3}}\prod _{j=J+1}^{2k+2}\|Iu_{N_j}\|_{4(2k-1),4(2k-1)} \\
&\lesssim \frac 1{m(N_2)N_2m(N_3)N_3\prod_{j=4}^Jm(N_j)N_j^{\frac{1}{2k-1}}}Z_I(t)^J \prod_{j=J+1}^{2k+2}\|I u_{N_j}\|_{4(2k-1),4(2k-1)}\\
&\lesssim N^{-2+}N_2^{0-}Z_I(t)^J\prod_{j=J+1}^{2k+2}\|Iu_{N_j}\|_{4(2k-1),4(2k-1)}.
\end{align*}
Applying interpolation, the bound for the $L_t^4L_x^8$ norm of $u$ that we assumed \eqref{eta}, and Bernstein, we bound
\begin{align}
\|Iu_{\leq 1}\|_{4(2k-1),4(2k-1)}
&\lesssim \|Iu_{\leq 1}\|_{L_t^4L_x^8}^{\frac{1}{2k-1}}\| I u_{\leq 1}\|_{L_t^{\infty}L_x^{16(k-1)}}^{\frac {2(k-1)}{2k-1}}\label{low}\\
&\lesssim \|Iu_{\leq 1}\|_{L_t^4L_x^8}^{\frac{1}{2k-1}}\| I u_{\leq 1}\|_{L_t^{\infty}L_x^{2k+2}}^{\frac {2(k-1)}{2k-1}}\\
&\lesssim \eta^{\frac{1}{2k-1}}\sup_{s\in[t_0,t]} E(Iu(s))^{\frac {k-1}{(2k-1)(k+1)}}. \nonumber
\end{align}
Thus,
\begin{align*}
B(N_1, \dots,  N_{2k+2})
\lesssim N^{-2+} N_2^{0-}\eta^{\frac {2k+2-J}{2k-1}} Z_I(t)^J \sup_{s\in[t_0,t]}E(I u(s))^{\frac {(k-1)(2k+2-J)}{(2k-1)(k+1)}}.
\end{align*}
The factor $N_2^{0-}$ allows us to sum in $N_2, \dots, N_J$.  This case contributes at most
$$
N^{-2+} \sum_{J=3}^{2k+2}\eta^{\frac {2k+2-J}{2k-1}}Z_I(t)^J \sup_{s\in[t_0,t]}E(I u(s))^{\frac {(k-1)(2k+2-J)}{(2k-1)(k+1)}}
$$
to the right-hand side of \eqref{al2}.
\\
\\
\noindent \textbf{Case $II_b$:} $N_1> 1$ and $N_2=\dots=N_{2k+2}=1$.
\\
\\
As $\sum_{i=1}^{2k+2}\xi_i=0$, we obtain $N_1\lesssim 1$ and thus, taking $N$ sufficiently large depending on $k$, we get
$$\symb=0.$$
\\
This case contributes zero to the right-hand side of \eqref{al2}.
\\
\\
\noindent \textbf{Case $II_c$:} $N_1> 1$ and $N_2>1=N_3=\dots=N_{2k+2}$.
\\
\\
As $\sum_{i=1}^{2k+2}\xi_i=0$, we must have $N_1\sim N_2$.  If $N_1\sim N_2\ll N$, then
$$
\symb=0
$$
and the contribution is zero.  Thus, we may assume $N_1\sim N_2 \gtrsim N$. Applying the Fundamental Theorem of Calculus,
\\
\begin{align*}
\Bigl|\symb\Bigr|&=\Bigl|1-\frac {m(\xi_2+\cdots+\xi_{2k+2})}{m(\xi_2)}\Bigr| \\
&\lesssim \Bigl|\frac {\nabla m(\xi_2)}{m(\xi_2)}\Bigr|\lesssim \frac {1}{N_2}.
\end{align*}
By the multilinear multiplier theorem,
\begin{align*}
B(N_1,\dots,N_{2k+2})
&\lesssim \frac {1}{N_2}\|\Delta I u_{N_1}\|_{4,4}\|Iu_{N_2}\|_{4,4}\prod_{j=3}^{2k+2}\|Iu_{N_j}\|_{4k,4k}\\
&\lesssim \frac {N_1}{N_2^2}\|\nabla I u_{N_1}\|_{4,4} \|\nabla I u_{N_2}\|_{4,4} \|Iu_{\leq 1}\|_{4k,4k}^{2k}\\
&\lesssim N^{-1+}N_2^{0-}Z_I(t)^2\|Iu_{\leq 1}\|_{4k,4k}^{2k}.
\end{align*}
The factor $N_2^{0-}$ allows us to sum in $N_1$ and $N_2$.  Using interpolation, \eqref{i1}, \eqref{eta}, and Bernstein, we estimate
\begin{align*}
\|Iu_{\le 1}\|_{4k,4k}
&\lesssim \|Iu_{\le 1}\|_{L_t^4L_x^8}^{\frac{1}{k}}\|Iu_{\le 1}\|_{L_t^{\infty}L_x^{8(k-1)}}^{1-\frac{1}{k}}\\
&\lesssim \eta^{\frac {1}{k}}\|Iu_{\le 1}\|_{L_t^{\infty}L_x^{2k+2}}^{1-\frac{1}{k}}\\
&\lesssim \eta^{\frac 1{k}}\sup_{s\in [t_0,t]} E(I u(s))^{\frac {k-1}{2k(k+1)}}.
\end{align*}
Thus, this case contributes at most
$$
N^{-1+}\eta^{2}Z_I(t)^2\sup_{s\in [t_0,t]} E(I u(s))^{\frac {k-1}{k+1}}
$$
to the right-hand side of \eqref{al2}.
\\
\\
\noindent \textbf{Case $II_d$:} $N_1> 1$ and there exists $J\ge 3$ such that $N_2\ge \dots \ge N_J> 1 =N_{J+1}=\dots=N_{2k+2}$.
\\
\\
To estimate the contribution of this case, we argue as for Case $I$; the only new ingredient is that the low frequencies are estimated via \eqref{low}.
This case contributes at most
\\
\begin{align*}
N^{-1+}\sum_{J=3}^{2k+2}\eta^{\frac {2k+2-J}{2k-1}}Z_I(t)^J \sup_{s\in[t_0,t]} E(Iu(s))^{\frac {(k-1)(2k+2-J)}{(2k-1)(k+1)}}
\end{align*}
to the right-hand side of \eqref{al2}.
\\
\\
Putting everything together, we get
\\
\begin{align}
\eqref{term1}
&\lesssim N^{-1+}Z_I(t)^{2k+2} + N^{-1+}\eta^{2}Z_I(t)^2\sup_{s\in [t_0,t]} E(I u(s))^{\frac {k-1}{k+1}} \notag\\
&\quad + N^{-1+}\sum_{J=3}^{2k+2}\eta^{\frac {2k+2-J}{2k-1}}Z_I(t)^J \sup_{s\in [t_0,t]} E(Iu(s))^{\frac {(k-1)(2k+2-J)}{(2k-1)(k+1)}}.\label{term1 est}
\end{align}
\\
\\
We turn now to estimating \textbf{\eqref{term2}}.  Again we decompose
$$
u:=\sum_{N\geq 1}P_N u
$$
with the convention that $P_1 u: =P_{\leq 1}u$.  Using this notation and symmetry, we estimate
\begin{align*}
\eqref{term2}\lesssim \sum_{\substack{N_1, \dots, N_{2k+2}\geq 1 \\ N_2\ge \cdots\ge N_{2k+2}}} C(N_1,\cdots,N_{2k+2}),
\end{align*}
where
\begin{align*}
C(& N_1,\cdots, N_{2k+2})\\
&:=\Bigl|\itxi\Bigl(\symb\Bigr)\\
&\qquad \widehat{\overline{P_{N_1}I(|u|^{2k}u)}}(\xi_1)\widehat{I u_{N_2}}(\xi_2)\cdots \widehat{\overline{Iu_{N_{2k+1}}}}
(\xi_{2k+1})\widehat{Iu_{N_{2k+2}}}(\xi_{2k+2})\, d\sigma(\xi)\,ds\Bigr|.
\end{align*}
\\
In order to estimate $C(N_1,\cdots, N_{2k+2})$ we make the observation that in estimating $B(N_1,\cdots, N_{2k+2})$,
for the term involving the $N_1$ frequency we only used the bound
\\
\begin{align}\label{1use}
\|P_{N_1}I\Delta u\|_{4,4}\lesssim N_1 \|\nabla Iu_{N_1}\|_{4,4}\lesssim N_1 Z_I(t).
\end{align}
\\
Thus, to estimate \eqref{term2} it suffices to prove
\\
\begin{equation}\label{al3}
\|P_{N_1}I(|u|^{2k}u)\|_{4,4}\lesssim Z_I(t)^{2k+1}+\eta \sup_{s\in [t_0,t]} E(Iu(s))^{\frac{k}{k+1}},
\end{equation}
\\
for then, arguing as for \eqref{term1} and substituting \eqref{al3} for \eqref{1use}, we obtain
\\
\begin{align*}
\eqref{term2}
&\lesssim N^{-1+}\Bigl(\zit^{2k+1}+ \eta^{2}Z_I(t)\sup_{s\in [t_0,t]} E(I u(s))^{\frac {k-1}{k+1}}\Bigr)\\
&\qquad \qquad\times \Bigl(\zit^{2k+1}+\eta \sup_{s\in [t_0,t]} E(I u(s))^{\frac {k}{k+1}}\Bigr)\\
&\quad +N^{-1+}\sum_{J=3}^{2k+2}\eta^{\frac{2k+2-J}{2k-1}}\zit^{J-1}\sup_{s\in [t_0,t]} E(I u(s))^{\frac {(k-1)(2k+2-J)}{(2k-1)(k+1}}\\
&\qquad\qquad \times \Bigl(\zit^{2k+1}+\eta \sup_{s\in [t_0,t]} E(I u(s))^{\frac {k}{k+1}}\Bigr).
\end{align*}
\\
Thus, we are left to proving \eqref{al3}.  Using \eqref{i1} and the boundedness of the Littlewood-Paley operators, and decomposing
$u:=u_{\leq 1}+u_{>1}$, we estimate
\\
\begin{align*}
\|P_{N_1}I(|u|^{2k}u)\|_{4,4}
&\lesssim \|u\|_{4(2k+1),4(2k+1)}^{2k+1}\\
&\lesssim\|u_{\le 1}\|_{4(2k+1),4(2k+1)}^{2k+1}+\|u_{>1}\|_{4(2k+1),4(2k+1)}^{2k+1}.
\end{align*}
\\
Applying interpolation, \eqref{eta}, and Bernstein, we estimate
\\
\begin{align*}
\|u_{\le 1}\|_{4(2k+1),4(2k+1)}^{2k+1}
&\lesssim\|u_{\le 1}\|_{L_t^4L_x^8}\|u_{\le 1}\|_{L_t^{\infty}L_x^{16k}}^{2k} \lesssim \eta \|u_{\le 1}\|_{L_t^{\infty}L_x^{2k+2}}^{2k}\\
&\lesssim \eta \sup_{s\in [t_0,t]} E(I u(s))^{\frac{k}{k+1}}.
\end{align*}
\\
Finally, by Sobolev embedding and \eqref{i2},
\\
\begin{align*}
\|u_{>1}\|_{4(2k+1),4(2k+1)}^{2k+1}
&\lesssim\||\nabla|^{\frac{2k}{2k=1}}u_{>1}\|_{4(2k+1),\frac{4(2k+1)}{4k+1}}^{2k+1}
\lesssim \zit^{2k+1}.
\end{align*}
Putting things together, we derive \eqref{al3}. This completes the proof of Proposition \ref{aclaw}.
\end{proof}
Now we will combine Proposition \ref{aclocallaw} and \ref{aclaw} and prove that the quantity $E(Iu)(t)$ is ``almost conserved''.
\begin{prop}\label{aclawfinal}
Let $s>\frac{1}{2k-1}$ and let $u$ be an $H_x^s$ solution to \eqref{nls} on the spacetime slab $[t_0,T]\times \R^2$ with $E(I_N u(t_0))\le 1$.
Suppose in addition that
\begin{align}\label{ms}
\|u\|_{L_{t \in [t_0,T]}^4L_{x}^8}\le \eta
\end{align}
for a sufficiently small $\eta>0$ (depending on $k$ and on $E(I_N u(t_0))$). Then, for $N$ sufficiently large (depending on $k$ and on $E(I_N u(t_0))$),
\\
\begin{equation}\label{eng growth}
\sup_{t\in [t_0,T]}E(I_N u(t))=E(I_N u(t_0))+ N^{-1+}.
\end{equation}
\end{prop}
\begin{proof}
Indeed, Proposition~\ref{aclawfinal}
follows immediately from Propositions \ref{aclocallaw} and \ref{aclaw}, if we establish
\begin{equation*}
Z_I(t)\lesssim 1 \quad \text{and} \quad \sup_{s\in [t_0,t]} E(I_N u(s))\lesssim 1 \quad \text{for all } t\in[t_0,T].
\end{equation*}
\\
As by assumption $E(I_N u(t_0))\leq 1$, it suffices to show that
\\
\begin{equation}\label{bdd1}
Z_I(t)\lesssim\|\nabla I_N u(t_0)\|_2 \quad \text{for all } t\in[t_0,T]
\end{equation}
and
\begin{equation}\label{bdd2}
\sup_{s\in [t_0,t]} E(I_N u(s))\lesssim E(I_N u(t_0)) \quad \text{for all } t\in[t_0,T].
\end{equation}
We achieve this via a bootstrap argument. We want to show
 that the set of times that those two properties hold is the set $[0, \infty)$. We define
\\
\begin{align*}
\Omega_1&:=\{t\in[t_0,T]:\, Z_I(t)\le C_1\|\nabla I_N u(t_0)\|_2,\\
    & \qquad \qquad \qquad \quad \sup_{s\in [t_0,t]} E(I_N u(s))\le C_2 E(I_N u(t_0))\}\\
\Omega_2&:=\{t\in[t_0,T]:\, Z_I(t)\le 2 C_1\|\nabla I_N u(t_0)\|_2,\\
    &\qquad \qquad \qquad  \quad \sup_{s\in [t_0,t]} E(I_N u(s)) \le 2C_2 E(I_Nu(t_0))\}.
\end{align*}
\\
 If we can prove that $\Omega_1$ is nonempty, open and closed then since the set $[0,\infty)$ is connected we must have that $\Omega_1=[0,\infty)$. 
Thus in order to run the bootstrap argument successfully, we need to check four things:
\\
i)\ $\Omega_1\neq \emptyset$. This is satisfied as $t_0\in \Omega_1$ if we take $C_1$ and $C_2$ sufficiently large.
\\
\\
ii)\ $\Omega_1$ is a closed set. This follows from Fatou's Lemma. 
\\
\\
iii)\ If $t\in \Omega_1$, then there exists $\epsilon>0$ such that $[t,t+\epsilon]\in \Omega_2$. This follows from the Dominated Convergence Theorem
combined with \eqref{zit control} and \eqref{energy increment}.
\\
\\
iv)\ $\Omega_2\subset \Omega_1$.  This follows from \eqref{zit control} and \eqref{energy increment} taking $C_1$ and $C_2$
sufficiently large depending on absolute constants (like the Strichartz constant) and choosing $N$ sufficiently
large and $\eta$ sufficiently small depending on $C_1$, $C_2$, $k$, and $E(I_N u(t_0))$.
\\
\\
The last two statements prove that $\Omega_1$ is open and the Proposition \ref{aclawfinal} is proved.
\end{proof}
Finally we are ready to prove Theorem 4. 
\begin{proof}
Given Proposition \ref{aclawfinal}, the proof of global well-posedness for \eqref{nls} is reduced to showing
\begin{equation}\label{ma bound}
\|u\|_{L_t^4L_x^8}\le C(\|u_0\|_{\hs}).
\end{equation}
This also implies scattering, as we will see later by an argument close to what we used to obtain Theorem 3. We have
 proved that
\begin{align}\label{ma control}
\|u\|_{L_t^4L_x^8}\lesssim \|u_0\|_2^{1/2} \|u\|_{L_t^\infty \dot H^{1/2}_x (\ir)}^{1/2}
\end{align}
on any spacetime slab $I \times \Bbb R^2$ on which the solution to \eqref{nls} exists and lies in $H_x^{1/2}$.
However, the $H_x^{1/2}$ norm of the solution is not a conserved quantity either, and in order to control it we must resort to
the $\hs$ bound on the solution.  As we remarked at the beginning of this section this will be achieved by controlling $\|Iu\|_{\dot{H}^1}$. 
Thus, in order to obtain a global Morawetz estimate, we need a global bound for $\|Iu\|_{\dot{H}^1}$. This will be done by patching together
 time intervals where the norm $\|u\|_{L_t^4L_x^8}$ is very small. This sets us up for a bootstrap argument.
\\
\\
Let $u$ be the solution to \eqref{nls}. As $E(Iu_0)$ is not necessarily small, we first rescale the solution such that the energy
of the rescaled initial data satisfies the conditions in Proposition \ref{aclawfinal}. By scaling,
$$
\ulam(x,t):=\lambda^{-\frac 1k} u({\lambda}^{-2} t,{\lambda^{-1}} x)
$$
is also a solution to \eqref{nls} with initial data
$$
u_0^{\lambda}(x):=\lambda^{-\frac1k}u_0({\lambda}^{-1} x).
$$
\\
By \eqref{i4} and Sobolev embedding for $s\geq 1-\frac{1}{k+1}$,
\\
\begin{align*}
\|\nabla Iu_0^{\lambda}\|_2
&\lesssim N^{1-s}\|u_0^{\lambda}\|_{\dhs}=N^{1-s}\lambda^{1-\frac{1}{k}-s}\|u_0\|_{\dhs},\\ \\
\|I_N u_0^{\lambda}\|_{2k+2}&\lesssim \|u_0^{\lambda}\|_{2k+2}=\lambda^{\frac 1{k+1}-\frac 1k}\|u_0\|_{2k+2}\lesssim\lambda^{\frac 1{k+1}-\frac 1k}\|u_0\|_{\hs}.
\end{align*}
\\
Since $s>1-\frac{1}{4k-3}>1-\frac{1}{k+1}>1-\frac{1}{k}$, choosing $\lambda$ sufficiently large (depending on $\|u_0\|_{\hs}$ and $N$) such that
\begin{equation}\label{lambdachoice}
N^{1-s} \lambda^{1 -\frac{1}{k} -s} \|u_0\|_{\hs} \ll 1 \quad \text{and}\quad \lambda^{\frac 1{k+1}-\frac 1k}\|u_0\|_{\hs}\ll 1,
\end{equation}
we get
$$
E(I_N u_0^{\lambda})\ll1.
$$
Thus
$$\lambda \sim N^{\frac{s-1}{1-s-\frac{1}{k}}}$$
We now show that there exists an absolute constant $C_1$ such that
\begin{equation}\label{rescaled ma}
\|\ulam\|_{L_t^4L_x^8}\le C_1\lambda^{\frac{3}{4}(1-\frac{1}{k})}.
\end{equation}
Undoing the scaling, this yields \eqref{ma bound}.
\\
\\
We prove \eqref{rescaled ma} via a bootstrap argument.  By time reversal symmetry, it suffices to argue for positive times only.  Define
\begin{align*}
\Omega_1&:=\{t\in[0,\infty):\, \|\ulam\|_{L_{t}^{4}L_{x}^{8}([0,t]\times \Bbb R^2)}\le C_1\lambda^{\frac{3}{4}(1-\frac{1}{k})},\} \\
\Omega_2&:=\{t\in[0,\infty):\, \|\ulam\|_{L_{t \in [0,t]}^{4}L_{x}^{8}([0,t]\times \Bbb R^2)}\le 2C_1\lambda^{\frac{3}{4}(1-\frac{1}{k})}.\}
\end{align*}
In order to run the bootstrap argument, we need to verify four things:
\\
1) $\Omega_1\neq\emptyset$.  This is obvious as $0\in\Omega_1$.
\\
2) $\Omega_1$ is closed.  This follows from Fatou's Lemma.
\\
3) $\Omega_2\subset\Omega_1$.
\\
4) If $T\in\Omega_1$, then there exists $\epsilon>0$ such that $[T,T+\epsilon)\subset \Omega_2$.  This is a consequence of the local
well-posedness theory and the proof of 3). We skip the details.
\\
\\
Thus, we need to prove 3).  Fix $T\in \Omega_2$; we will show that in fact, $T\in\Omega_1$.
By \eqref{ma control} and the conservation of mass,
\begin{align*}
\|\ulam\|_{L_{t}^{4}L_{x}^{8}([0,t]\times \Bbb R^2)}
&\lesssim \|u_0^{\lambda}\|_2^{\frac 12}\|\ulam\|^{\frac 12}_{L_t^\infty \dot H_x^{1/2}([0,T]\times\Bbb R^2)}\\
&\lesssim \lambda^{\frac{1}{2}(1-\frac{1}{k})} C(\|u_0\|_2) \|\ulam\|_{L_t^\infty \dot H_x^{1/2}([0,T]\times\Bbb R^2)}^{\frac 12}.
\end{align*}
To control the factor $\|\ulam\|_{L_t^\infty \dot H_x^{1/2}([0,T]\times\Bbb R^2)}$, we decompose
$$
\ulam(t):=P_{\le N}\ulam(t)+P_{>N}\ulam(t).
$$
To estimate the low frequencies, we interpolate between the $\lxt$ norm and the $\ho$ norm and use the fact that $I$ is
the identity on frequencies $|\xi|\le N$
\begin{align*}
\|P_{\le N}\ulam(t)\|_{\dot H_x^{1/2}}
&\lesssim \|P_{\le N}\ulam(t)\|_2^{\frac 12}\|P_{\le N}\ulam(t)\|_{\ho}^{\frac 12}\\
&\lesssim\lambda^{\frac{1}{2}(1-\frac{1}{k})}C(\|u_0\|_2)\|I_N \ulam(t)\|_{\ho}^{\frac 12}.
\end{align*}
To control the high frequencies, we interpolate between the $\lxt$ norm and the $\dhs$ norm and use Lemma \ref{basic property} and
 the relation between $N$ and $\lambda$ to get
\begin{align*}
\|P_{>N}\ulam(t)\|_{\dot H^{1/2}_x}
&\lesssim \|P_{>N}\ulam(t)\|_{\lxt}^{1-\frac 1{2s}}\|P_{>N}\ulam(t)\|_{\dhs}^{\frac 1{2s}}\\
&\lesssim \lambda^{(1-\frac 1{2s})(1-\frac{1}{k})}N^{\frac{s-1}{2s}}\|I\ulam(t)\|_{\ho}^{\frac1{2s}}\\
&\lesssim \lambda^{\frac{1}{2}-\frac{1}{k}}\|I\ulam(t)\|_{\ho}^{\frac 1{2s}}.
\end{align*}
Collecting all these estimates, we get
\begin{align*}
\|\ulam\|_{L_{t}^{4}L_{x}^{8}([0,t]\times \Bbb R^2)}
&\lesssim \lambda^{\frac{3}{4}(1-\frac{1}{k})} C(\|u_0\|_2)
   \sup_{t\in [0,T]}\bigl(\|\nabla I \ulam(t)\|_2^{\frac 14}+\|\nabla I\ulam(t)\|_2^{\frac 1{4s}}\bigr).
\end{align*}
Thus, taking $C_1$ sufficiently large depending on $\|u_0\|_2$, we obtain $T\in \Omega_1$, provided
\begin{equation}\label{bdd kinetic}
\sup_{t\in [0,T]}\|\nabla I\ulam(t)\|_2\le 1.
\end{equation}
We now prove that $T\in\Omega_2$ implies \eqref{bdd kinetic}.  Indeed, let $\eta>0$ be a sufficiently small constant like in
Proposition~\ref{aclawfinal} and divide $[0,T]$ into
$$
L\sim \biggl(\frac{\lambda^{\frac{3}{4} (1-\frac 1k)}}{\eta}\biggr)^4
$$
sub-intervals $I_j=[t_j,t_{j+1}]$ such that,
$$
\|\ulam\|_{L_{t}^{4}L_{x}^{8}(I_j\times \Bbb R^2)}\le \eta.
$$
Applying Proposition \ref{aclawfinal} on each of the sub-intervals $I_j$,
we get
\\
$$\sup_{t\in [0,T]} E(I_N\ulam(t))\le E(I_Nu_0^{\lambda})+ E(I_N u_0^{\lambda})LN^{-1+}.$$
To maintain small energy during the iteration, we need
\begin{align*}
LN^{-1+}\sim\lambda^{3(1-\frac 1k)}N^{-1+}\ll 1,
\end{align*}
which combined with \eqref{lambdachoice} leads to
$$
\biggl(N^{\frac{1-s}{s-1+\frac{1}{k}}}\biggr)^{3(1-\frac{1}{k})}N^{-1+}\le c(\|u_0\|_{\hs})\ll 1.
$$
This may be ensured by taking $N$ large enough (depending only on $k$ and $\|u_0\|_{H^s(\R^)}$), provided that
$$
s>s(k):=1-\frac{1}{4k-3}.
$$
As can be easily seen, $s(k)\to 1$ as $k\to \infty$.

This completes the bootstrap argument and hence \eqref{rescaled ma}, and moreover \eqref{ma bound}, follows.  Therefore \eqref{bdd kinetic}
holds for all $T\in\R$ and the conservation of mass and Lemma~\ref{basic property} imply
\\
\begin{align*}
\|u(T)\|_{\hs}&\lesssim \|u_0\|_{\lxt}+\|u(T)\|_{\dhs}\\
&\lesssim\|u_0\|_{\lxt}+\lambda^{s-(1-\frac{1}{k})}\|\ulam(\lambda^2T)\|_{\dhs}\\
&\lesssim \|u_0\|_{\lxt}+\lambda^{s-(1-\frac{1}{k})}\|I\ulam(\lambda^2T)\|_{H_x^1}\\
&\lesssim\|u_0\|_{\lxt}+\lambda^{s-(1-\frac{1}{k})}(\|\ulam(\lambda^2T)\|_{\lxt}+\|\nabla I\ulam(\lambda^2T)\|_{\lxt})\\
&\lesssim\|u_0\|_{\lxt}+\lambda^{s-(1-\frac{1}{k})}(\lambda^{1-\frac{1}{k}}\|u_0\|_{\lxt}+1)\\
&\lesssim  C(\|u_0\|_{\hs})
\end{align*}
for all $T\in \R$.  Hence,
\begin{align}\label{hsbdd}
\|u\|_{L_t^{\infty}\hs}\le C(\|u_0\|_{\hs}).
\end{align}
\\
Finally, we prove that scattering holds in $\hs$ for $s>s_k$.  The construction of the wave operators is standard and follows
 by a fixed point argument (see \cite{tc}). Here we show only asymptotic completeness.

The first step is to upgrade the global Morawetz estimate to global Strichartz control.
Let $u$ be a global $\hs$ solution to \eqref{nls}.  Then $u$ satisfies \eqref{ma bound}.
Let $\delta>0$ be a small constant to be chosen momentarily and split $\R$ into $L=L(\|u_0\|_{\hs})$ sub-intervals $I_j=[t_j, t_{j+1}]$ such that
$$
\|u\|_{L_{t}^{4}L_{x}^{8}(I_j\times \Bbb R^2)}\le \delta.
$$
By Lemma \ref{linstr}, \eqref{hsbdd}, and the fractional chain rule, \cite{cw}, we estimate
\begin{align*}
\|\langle\nabla\rangle^s u\|_{S^0(I_j)}
&\lesssim \|u(t_j)\|_{\hs}+\|\langle\nabla\rangle^s\bigl(|u|^{2k}u\bigr)\|_{L_{t,x}^{4/3}(I_j\times\Bbb R^2)}\\
&\lesssim C(\|u_0\|_{\hs})+\|u\|_{L_{t,x}^{4k}}^{2k}\|\langle \nabla\rangle^s u\|_{L_{t,x}^4(I_j\times\Bbb R^2)},
\end{align*}
while by H\"older and Sobolev embedding,
\begin{align*}
\|u\|_{L_{t,x}^{4k}(I_j\times\Bbb R^2)}^{2k}
&\lesssim \|u\|_{L_{t}^{4}L_{x}^{8}(I_j\times \Bbb R^2)}^{\frac {2k}{2k-1}}\|u\|_{L_{t}^{8k}L_{x}^{\frac{16k(k-1)}{3k-2}}(I_j\times\Bbb R^2)}^{\frac{4k(k-1)}{2k-1}}\\
&\lesssim \delta^{\frac {2k}{2k-1}}\||\nabla|^{\frac{8k^2-13k+4}{8k^2-8k}}u\|_{L_t^{8k}L_x^{\frac {8k}{4k-1}}(I_j\times\Bbb R^2)}^{\frac{4k(k-1)}{2k-1}}\\
&\lesssim \delta^{\frac {2k}{2k-1}}\|\langle\nabla\rangle^s u\|_{S^0(I_j)}^{\frac{4k(k-1)}{2k-1}}.
\end{align*}
The last inequality follows form the fact that for any $k \geq 2$ we have that $$s_{k}=1-\frac{1}{4k-3}>\frac{8k^2-13k+4}{8k^2-8k}.$$ Therefore,
$$
\|\langle\nabla\rangle^s u\|_{S^0(I_j)}
\lesssim C(\|u_0\|_{\hs})+\delta^{\frac {2k}{2k-1}}\|\langle\nabla\rangle^s u\|_{S^0(I_j)}^{1+\frac{4k(k-1)}{2k-1}}.
$$
A standard continuity argument yields
$$
\|\langle\nabla\rangle^su\|_{S^0(I_j)}\le C(\|u_0\|_{\hs}),
$$
provided we choose $\delta$ sufficiently small depending on $k$ and $\|u_0\|_{\hs}$.  Summing over all sub-intervals $I_j$, we obtain
\begin{equation}\label{s bound}
\|\langle\nabla\rangle^s u\|_{S^0(\R)}\le C(\|u_0\|_{\hs}).
\end{equation}

We now use \eqref{s bound} to prove asymptotic completeness, that is, there exist unique $u_{\pm}$ such that
\begin{equation}\label{limit}
\lim_{t\to \pm\infty}\|u(t)-e^{it\Delta}u_{\pm}\|_{\hs}=0.
\end{equation}
\\
Arguing as in Section 4, it suffices to see that
\begin{align}\label{lg}
\Bigl\|\int_{t}^\infty e^{-is\Delta}\bigl(|u|^{2k}u\bigr)(s)\,ds\Bigr\|_{\hs}\to 0 \quad \text{as } t\to \infty.
\end{align}
\\
The estimates above yield
\begin{align*}
\Bigl\|\int_{t}^\infty e^{-is\Delta}\bigl(|u|^{2k}u\bigr)(s)\,ds\Bigr\|_{\hs}
\lesssim \|u\|_{L_{t}^{4}L_{x}^{8}(I_j\times \Bbb R^2)}^{\frac {2k}{2k-1}}\|\langle \nabla \rangle^s u\|_{S^0([t,\infty]\times\R)}^{1+\frac{4k(k-1)}{2k-1}}.
\end{align*}
\\
Using \eqref{ma bound} and \eqref{s bound} we derive \eqref{lg}. This concludes the proof of Theorem 5.
\end{proof}


\begin{thebibliography}{100}

\bibitem{bl}
J. Bergh, J. L\"ofstr\"om,
{\it Interpolation Spaces,} Springer Verlag 1976.
\bibitem{jb1}
J. Bourgain, {\it Refinements of Strichartz' inequality and applications to
2D-NLS with critical nonlinearity,}
International Mathematical Research Notices, 5 (1998), 253--283.

\bibitem{jb2}
J. Bourgain.
{\it Global solutions of nonlinear {S}chr\"odinger equations,} American Mathematical Society, Providence, RI, 1999.

\bibitem{tc} T. Cazenave,
{\it Semilinear Schr\"oodinger equations,} CLN 10, eds: AMS, 2003.

\bibitem{cw} M. Christ, M. Weinstein, {\it Dispersion of small amplitude solutions of the
generalized Korteweg-de Vries equation}, J. Funct. Anal. \textbf{100} (1991), 87-109.

\bibitem{cmfourier}
R. Coifman, Y.  Meyer, {\it Commutateurs d'int\'egrales singuli\'eres et op\'erateurs multilin\'eaires},
Ann. Inst. Fourier (Grenoble) \textbf{28} (1978), 177--202.

\bibitem{cmaster}
R. Coifman, Y. Meyer,  {\it Au del\'a des op\'erateurs pseudo-diff\'erentiels},
Ast\'erisque \textbf{57}, Soci\'et\'e Math\'ematique de France, Paris, 1978.

\bibitem{col} J. Colliander (joint with M. Grillakis and N. Tzirakis), {\it The interaction Morawetz inequality on $\Bbb R^2$}, 
Workshop "Nonlinear Waves and Dispersive Equations", Oberwolfach Report No 44/2007.

\bibitem{cgt} J. Colliander, M. Grillakis, and N. Tzirakis,
{\it Improved interaction Morawetz inequalities for the cubic nonlinear Schr\"odinger equation in 2d,} 
to appear in International Mathematics Research Notices.

\bibitem{chvz} J. Colliander, J. Holmer, M. Visan and X. Zhang,
{\it Global existence and scattering for rough solutions to generalized
nonlinear Schr\"odinger equations on $\mathbb R$,}
Preprint (2006). ({\tt{math.AP/0612452}})

%\bibitem{ckstt2} J. Colliander, M. Keel, G. Staffilani, H. Takaoka, T. Tao,
%{\it Almost conservation laws and global rough solutions to a nonlinear Schr\"odinger
%equation}, Math. Research Letters, 9 (2002), 659-682.


\bibitem{ckstt4} J. Colliander, M. Keel, G. Staffilani, H. Takaoka and
T. Tao,
 {\it Global existence and scattering for rough solutions to a nonlinear Schr\"odinger
 equations on $\R^{3}$ }, C.P.A.M. \textbf{57} (2004), no. 8, 987--1014.


\bibitem{ckstt5} J. Colliander, M. Keel, G. Staffilani, H. Takaoka and T. Tao,
{\it Global well-posedness and scattering in the energy space for the critical nonlinear Schr\"odinger equation in $\mathbb R^{3}$,}
Annals Math. \textbf{167} (2008), 767--865. ({\tt{math.AP/0402129}})

%\bibitem{CKSTT:Angles} J. Colliander, M. Keel, G. Staffilani,
%  H. Takaoka and T. Tao, {\it {Global well-posednesss for the cubic
%    nonlinear Schr\"odinger equation in $H^s (\R^2)$ for $s > 1/2$}},
%  preprint, 2007.

\bibitem{dps} D. De Silva, N. Pavlovic, G. Staffilani, and N. Tzirakis,
{\it Global well-posedness and polynomial bounds for the defocusing nonlinear Schr\"odinger equation in 1d,} 
to appear in Commun. Partial Differential Equations.

\bibitem{fg}
Y. Fang and M. Grillakis, {\it On the global existence of rough
solutions of the cubic defocusing Schr\"odinger equation in
$\R^{2+1}$,} to appear in JHDE.

\bibitem{gv} J. Ginibre and G. Velo, 
{\it The global Cauchy problem for the nonlinear Schr\"oodinger equation,} H. Poincar\'e
Analyse Non Lin\'eaire, 2 (1985), 309-327.

\bibitem{gv1} J. Ginibre and G. Velo, 
{\it Scattering theory in the energy space for a class of nonlinear Schr\"odinger equations,} J. Math. Pures Appl. 64 (1985), 363-401.

\bibitem{kt} M. Keel and T. Tao,
{\it Endpoint Strichartz estimates,} H. Poincar\'e Analyse non
Lin\'eaire, 120 (1998), 955-980.

%\bibitem{ktv}  R. Killip, T. Tao, and M. Visan,
%{\it The cubic nonlinear Schr\"odinger equation in two dimensions with radial data,} preprint.

\bibitem{ls} J. E. Lin and W. A. Strauss, {\it Decay and
    scattering of solutions of a nonlinear Schr\"odinger equation},
  J. Funct. Anal., \textbf{30}:2, (1978), 245--263.

\bibitem{cm} C. Morawetz,
{\it Time decay for the nonlinear Klein-Gordon equation,} Proc. Roy. Soc. A, 306 (1968), 291-296.

\bibitem{kn} K. Nakanishi, {\it Energy scattering for nonlinear Klein-Gordon and Schr\"odinger equations in spatial dimensions 1 and 2,} 
J. Funct. Anal. 169 (1999), 201-225.

\bibitem{pv1} F. Planchon (joint with L. Vega) \emph{Bilinear Virial Identities and Applications}, Workshop "Nonlinear Waves and Dispersive Equations", Oberwolfach Report No 44/2007.

\bibitem{pv} F. Planchon and L. Vega, \emph{Bilinear Virial
    Identities and Applications}, to appear,  Annales
Scientifiques de l'\'Ecole Normale Sup\'erieure. ({\tt{math.AP/0712.4076}})

\bibitem{es} E. M. Stein,
{\it Harmonic Analysis: Real variable Methods, Orthogonality and Oscillatory integrals,} Princeton Univ. Press, Princeton (1993).

\bibitem{tt} T. Tao,
{\it Nonlinear dispersive equations. Local and global analysis} CBMS 106, eds: AMS, 2006.

%\bibitem{tv}  T. Tao, M. Visan, and X. Zhang,
%{\it Global well-posedness and scattering for the mass-critical nonlinear
%Schr\"odinger equations for radial data in high dimensions,} to appear in Duke Math. J.

\bibitem{vis1}  M. Visan,
{\it The defocusing energy-critical nonlinear Schr\"odinger equation in higher dimensions,} to appear in Duke Math. J.

\bibitem{mv}  M. Visan and  X. Zhang,
{\it Global well-posedness and scattering for a class of nonlinear
Schr\"odinger equations below the energy space,} preprint.

%\bibitem{tvz}  T. Tao, M. Visan, and X. Zhang
%{\it Global well-posedness and scattering for the mass-critical nonlinear
%Schr\"odinger equations for radial data in high dimensions,} preprint,
%2006. ({\tt{math.AP/0609692 }})

\end{thebibliography}
\end{document}